\documentclass[numbers,webpdf]{ima-authoring-template}%

\graphicspath{{Fig/}}

\usepackage{float}
\theoremstyle{thmstyletwo}%
\newtheorem{theorem}{Theorem}%  meant for continuous numbers
\newtheorem{proposition}[theorem]{Proposition}%

\newcommand{\roma}{\mathrm{I}}
\newcommand{\romb}{\mathrm{II}}

\newcommand{\dx}{\,\mathrm{d}x}
\newcommand{\ds}{\,\mathrm{d}s}
\newcommand{\prnt}[1]{\left( #1 \right)}

\newcommand{\norm}[1]{\left\|#1\right\|}

\newcommand{\normL}[2]{\norm{#1}_{L^2\prnt{#2}}}

\newcommand{\normHp}[3]{\norm{#1}_{H^{#3}\prnt{#2}}}

\newcommand{\Cov}{C_{\mathrm{ov}}}
\newcommand{\Cw}{{\rm C}_{\mathrm{weak}}}
\newcommand{\Ce}{{\rm C}_{\mathrm{est}}}
\newcommand{\Const}[1]{{\rm C}_{\mathrm{#1}}}

\newcommand{\Co}{{\rm C}_{\mathrm{ap}}(k)}
\newcommand{\Cpoin}[1]{{\rm C}_{\operatorname{P}}(#1)}
\newcommand{\Cpoinn}[2]{{\rm C}^{#2}_{\operatorname{P}}(#1)}

\newcommand{\RR}{\mathcal{R}}
\newcommand{\op}{\operatorname}
\newtheorem{corollary}{Corollary}
\newtheorem{assumption}{Assumption}
\newtheorem{lemma}{Lemma}
\usepackage{soul,ulem}
\setstcolor{red}
\allowdisplaybreaks

\usepackage{comment}
\usepackage{mathtools}
\usepackage{booktabs} 
\usepackage{multirow} 
\usepackage{adjustbox}
\usepackage{subfigure}
\usepackage{algorithm}
\usepackage{algpseudocode}
\usepackage{tikz}
\usetikzlibrary{arrows.meta,calc}

\newtheorem{remark}{Remark}%

\numberwithin{equation}{section}

\begin{document}

\DOI{DOI HERE}
\copyrightyear{2021}
\vol{00}
\pubyear{2021}
\access{Advance Access Publication Date: Day Month Year}
\appnotes{Paper}
\copyrightstatement{Published by Oxford University Press on behalf of the Institute of Mathematics and its Applications. All rights reserved.}
\firstpage{1}

%\subtitle{Subject Section}

\title[EMs-HS for High-frequency Helmholtz problems]{Discrete Edge Multiscale Space based Hybrid Schwarz Preconditioner for High-Frequency Helmholtz Problems}

\author{Shubin Fu
\address{\orgdiv{Eastern Institute for Advanced Study},
\orgname{Eastern Institute of Technology},
\orgaddress{\state{Ningbo}, \postcode{315200}, \state{Zhejiang}, \country{P. R. China}}}
}

\author{Shihua Gong
\address{\orgdiv{School of Science and Engineering},
\orgname{The Chinese University of Hong Kong, Shenzhen},
\orgaddress{\state{Shenzhen}, \postcode{518172}, \state{Guangdong}, \country{P. R. China}}}
}

\author{Guanglian Li*
\address{\orgdiv{Department of Mathematics},
\orgname{The University of Hong Kong},
\orgaddress{\street{Pok Fu Lam Road}, \country{Hong Kong SAR, P. R. China}}}
}

\author{Yueqi Wang
\address{\orgdiv{Department of Mathematics},
\orgname{Purdue University},
\orgaddress{\street{610 Purdue Mall}, \state{West Lafayette}, \postcode{47907}, \state{IN}, \country{USA}}}
}

\authormark{Fu et al.}

\corresp[*]{Corresponding author: \href{email:email-id.com}{lotusli@math.hku.hk}}

\received{Date}{0}{Year}
\revised{Date}{0}{Year}
\accepted{Date}{0}{Year}

%\editor{Associate Editor: Name}

\abstract{In this work, we develop a novel hybrid Schwarz method, termed as edge multiscale space based hybrid Schwarz (EMs-HS), for solving the Helmholtz problem with large wavenumbers. The problem is discretized using $H^1$-conforming nodal finite element methods on meshes of size $h$ decreasing faster than $k^{-3/2}$ such that the discretization error can be bounded by the data uniformly with respect to $k$. 
EMs-HS consists of a one-level Schwarz preconditioner (RAS-imp) and a coarse solver in a multiplicative way. The RAS-imp preconditioner solves local problems on overlapping subdomains with impedance boundary conditions in parallel, and combines the local solutions using  partition of unity. On the PDE level, this coarse space corresponds to an edge multiscale space proposed in {[S. Fu, G. Li, R. Craster, and S. Guenneau, \emph{Multiscale Model. Simul.}, 19(4):1684–1709,
2021]}. The approximation properties of this coarse space to the discrete solutions from the primal and adjoint problems are derived by a novel conforming Galerkin finite element method based upon a nonstandard variational formulation. This work represents the first theoretical analysis of an edge multiscale space at the discrete level.
We prove that the EMs-HS preconditioner leads to a convergent fixed-point iteration uniformly for large wavenumbers, by rigorously analyzing the approximation properties of the coarse space to the global Helmholtz harmonic extension part and to the solution of the adjoint problem. We provide extensive numerical results on the sharpness of the theoretical findings and also demonstrate the method on challenging heterogeneous models. }

% \boxedtext{
% \begin{itemize}
% \item Key boxed text here.
% \item Key boxed text here.
% \item Key boxed text here.
% \end{itemize}}

\maketitle

\section{Introduction}

The efficient solution of the Helmholtz problem with large wavenumbers on modern multiprocessor computers is  of huge importance in many practical applications, e.g., acoustics, geophysics and electromagnetics, and has long been considered one of the most challenging problems in scientific computing. Helmholtz problems are numerically challenging because (a) the solutions are often highly oscillatory and very fine discretizations are needed to resolve them, leading to large system matrices; (b) the system matrices are highly indefinite and non-normal, and standard iterative methods are not reliable; (c) the propagative nature of the problem implies that the presence of a source at any point can have considerable effect far away, thus inhibiting the performance of parallel algorithms which rely on localization techniques. Each of these essential difficulties gets worse as the wavenumber $k$ increases. 

In recent years there have been significant developments in the theory of one-level Schwarz methods for high-frequency Helmholtz problems. Since the Galerkin projection of discrete Helmholtz problems into a subdomain might be singular, one-level Schwarz methods usually require stabilization, e.g., shifted-Laplacian technique \cite{erlangga2006novel, erlangga2004class}. One-level methods with symmetrized weighted prolongation and restriction by partition of unity lead to good field of values for preconditioning absorptive Helmholtz problems  \cite{graham2020domain,gong2021domain}. There are also one-level methods using wave-type truncation techniques to construct local problems \cite{kimn2007restricted,engquist2011sweepinga}, e.g. RAS-imp preconditioners \cite{gong2022convergence,gong2023convergence} using absorbing boundary condition  and the RAS-PML preconditioners \cite{galkowski2024convergence} using perfectly matched layers, in which the local solutions are combined using partition of unity. With stable local problems, these one-level methods are power contractive. Given a fixed setting of domain decomposition with generous overlap, they can be $k$-robust even without a coarse space, but the drawback is that the subdomain problems are fairly costly and the algorithms are not scalable.

To obtain a scalable domain decomposition algorithm (i.e., weakly dependent on the number of subdomains), one usually needs a two-level method by introducing a coarse space that can capture the global information and effectively transfer it among subdomains. Two-level methods utilize a standard algorithmic strategy for scalable solutions of large systems of equations and have demonstrated great success in steady-state diffusion, elasticity and fluid problems. The classical Schwarz theory \cite{toselli2006domain} can predict the performance of two-level methods for these self-adjoint problems. One basic tool is the fictitious space lemma that requires constructing a stable decomposition for each function in the discrete solution space. Whereas the coarse function space provides a coarse approximation, the local function spaces resolve the high-frequency remainder, which should be relatively small in the $L^2$-norm compared to its $H^1$-norm. The classical Schwarz theory \cite{cai1992domain,xu1992new} can also  predict the performance of certain two-level domain decomposition algorithms for solving non-symmetric and indefinite problems when the indefinite part of the operator is a small perturbation of the elliptic part. Unfortunately, for the high-frequency Helmholtz equation, the indefinite part is amplified by the square of the wavenumber $k^2$ and the classical theory gives very pessimistic estimates  \cite{graham2017domain}.

Thus, it is natural to ask whether the theory of \cite{graham2020domain,gong2021domain,gong2022convergence,gong2023convergence,galkowski2024convergence} can be improved if a coarse space were added. For example, can the addition of a coarse space speed up the information exchange among subdomains? Or can the addition of a coarse space reduce the requirement of the overlapping size in the domain decompositions? For the conventional coarse spaces (i.e., piecewise polynomials on a coarse grid), the theory of coarse approximation for Helmholtz is inextricably linked to the theory of finite element approximation of the Helmholtz problem itself; see Remark \ref{rm:mesh-resolution} for more discussions. In particular, if the diameter of the coarse mesh is not small enough, there is no guarantee that a coarse projection exists at all, let alone any accuracy. The most well-known coarse spaces for Helmholtz problems with heterogeneous coefficients are DtN coarse space, GenEO coarse space, Localized Orthogonal Decomposition (LOD) and multiscale spectral generalized finite element method (MS-GFEM) \cite{MR4299039,MR4621835,MR4890903,Graham-geneo-2024,MR4980418, MR4992072,MR5013301}. Hence it is crucial to construct a coarse space that scales weakly with respect to the wavenumber, while ensuring a stable coarse projection. In the context of the Helmholtz problem with large wavenumbers, ideally the coarse space should overcome the so-called pollution effect with very mild resolution condition over the coarse grid, while maintain coarse grid accuracy.
 
In this work, we propose a coarse space $V_{h,0}$ using the discrete edge multiscale (EMs) spaces \cite{MR3939320,MR3980476,MR4860976}. EMs spaces are one class of multiscale ansatz spaces \cite{MR2718268,MR2728702,MR3834684,MR3614010,MR4860976,MR4601686}, and were originally developed for approximating the solutions of elliptic PDEs with strongly oscillating coefficients. EMs spaces arise from the global splitting of the solution by a global bubble part and a global Helmholtz harmonic extension part, realized by a local splitting in each subdomain. Since the global bubble part can be recovered locally in parallel, the EMs space is constructed to approximate the global Helmholtz harmonic extension part by solving a homogeneous Helmholtz equation with hierarchical bases up to level $\ell$ as the Dirichlet boundary data in each subdomain and then use partition of unity functions to glue the local bases together as global multiscale bases. The dimension of $V_{h,0}$ is $H^{-d}2^{\ell(d-1)}$ with $\ell\leq \log_2\frac{H}{h}$, and EMs-HS is compatible with 
any sequence of overlapping subdomains that decomposes $\Omega$ and its associated partition of unity.
Inspired by \cite{MR4890903}, we establish in Theorem \ref{prop:wavelet-basedconv} the convergence of the proposed  EMs-HS preconditioner, i.e., the contraction property of the coarse correction for the global harmonic extension part, using the standard Schatz argument under very mild conditions on overlapping size $\delta$, the maximum diameter of the subdomains $\bar{d}$ and the level parameter $\ell$ (see Theorem \ref{prop:wavelet-basedconv}). 
This is achieved as follows. In Corollary \ref{thm:proj}, we derive the local approximation properties of the coarse space, by transferring the approximation properties over the coarse skeleton to its interior by means of the transposition method \cite{MR0350177} (see also \cite[Appendix A]{MR4860976}) and analyzing the conforming Galerkin finite element approximation error to a nonstandard variational formulation in Section \ref{sec:very-weak}. Note that Section \ref{sec:very-weak} presents a self-contained analysis of Galerkin finite element approximation errors using a nonstandard variational formulation, a framework of independent theoretical interest. Then in Lemmas \ref{thm:last-app-harmonic} and \ref{thm:last}, we derive its approximation properties to the global Helmholtz harmonic extension part and the solution to the adjoint problem by means of representing the global error in terms of local errors.
Theorem \ref{prop:wavelet-basedconv} indicates that in the optimal case with uniform meshes, the proposed iterative solvers only require the dimension of the coarse spaces to be about $\min\{k^{3d-2},k^{-1/2}h^{-d}\}$
in order to achieve convergence for the Helmholtz equation with large wavenumbers. Extensive numerical experiments for both 2-d and 3-d cases with homogeneous coefficients and heterogeneous coefficients demonstrate that $\ell=1$ is sufficient for rapid convergence.

The rest of the paper is structured as follows. We describe the basics on the weak formulation and FEM approximation in Section \ref{sec:prep}. In Section \ref{sec:hybrid-2}, we present  the one-level RAS-imp preconditioner and the hybrid Schwarz preconditioner. Then in Section \ref{sec:coarse-space}, we give the construction of the coarse space, and in Section \ref{sec:converge} prove the convergence of the hybrid Schwarz preconditioner. In Section \ref{sec:num}, we present 2-d and 3-d numerical tests to complement the theoretical findings. Throughout, We employ standard notation for $L^p$ or $L^\infty$ spaces, the Sobolev space $W^{m,p}(D;\mathbb{C})$, and their associated (semi-) norms.
We define the complex-valued space $V:=H^{1}(\Omega;\mathbb{C}):={W^{1,2}(\Omega;\mathbb{C})}$, equipped with the $k$-weighted norm
\[
\norm{v}_{V}:=\sqrt{\|\nabla v\|_{L^2(\Omega)}^2+k^2\|v\|_{L^2(\Omega)}^2},
\]
and similarly define the complex-valued space $V(D):=H^{1}(D;\mathbb{C})$ also equipped with the $k$-weighted norm
$\norm{\cdot}_{1,k,D}$ for any $D\subset \Omega$.
We define $(w,v)_{D}=\int_Dw\overline{v}{\rm d}x$ the inner product on $L^2(D;\mathbb{C})$ and denote $\|\cdot\|_{L^2(D)}$ its induced norm
%denote by $(\cdot,\cdot)_{D}$ and $\|\cdot\|_{L^2(D)}$ the $L^2(D;\mathbb{C})$ inner product and its induced norm 
for any $D\subset\Omega$, $\Re$, $\Im$ and $\overline{\cdot}$ the real part, imaginary part, and conjugate of a complex number. The notation $a\lesssim b$ means that there exists a constant $C>0$, independent of the parameters of interest ($k,h,\bar{d},\delta$), but it may depend on the triangulation and the domain $\Omega$, such that $a\leq Cb$. We also write $a\simeq b$ if $a \lesssim b$ and $b \lesssim a$.

\section{The Helmholtz equation and its finite element approximation}\label{sec:prep}
Let $\Omega\subset \mathbb{R}^d $ ($d=2,3$) be a bounded convex polygonal/polyhedron domain, $f\in (H^1(\Omega;\mathbb{C}))^\prime$ and $g\in H^{-1/2}(\partial\Omega;\mathbb{C})$. Consider the Helmholtz equation subject to the impedance boundary condition
\begin{equation}
    \label{hel_cont}
   \left\{  \begin{split}
    -\Delta u-k^2u&=f,\quad\mbox{in }\Omega,\\
    \frac{\partial u}{\partial n}+\mathrm{i}ku&=g,\quad\mbox{on }\partial \Omega, 
     \end{split}\right.
\end{equation}
where without loss of generality, we assume the wavenumber $k\geq k_0>0$. 

The variational formulation for \eqref{hel_cont} reads: find $u\in V$ such that 
\begin{equation}
\label{var_form}
    a(u,v)=F(v), \quad\forall v\in V,
\end{equation}
where the sesquilinear form $a(\cdot,\cdot): V\times V\to \mathbb{C}$ and the conjugate linear form $F: V\to \mathbb{C}$ are defined by
\begin{align*}
  a(v_1,v_2)&:= \int_\Omega\nabla v_1\cdot\nabla\overline{v}_2\mathrm{d}x-k^2\int_\Omega v_1\overline{v}_2\;\mathrm{d}x+ \mathrm{i}k\int_{\partial \Omega} v_1\overline{ v}_2 \;\mathrm{d}s, \quad\forall v_1,v_2\in V, \\
F(v)&:=\int_\Omega f\overline{v}\;\mathrm{d}x+\int_{\partial \Omega} g\overline{v}\;\mathrm{d}s,\quad\forall v\in V.   
\end{align*}
The following properties of the sesquilinear form $a(\cdot,\cdot)$ hold \cite[Theorem 3.2 and Corollary 3.3]{MR2812565}.
\begin{theorem}[Properties of the sesquilinear form $a(\cdot,\cdot): V\times V\to \mathbb{C}$]\label{them:sesquilinear}
The following properties hold,
\begin{itemize}
\item[1.] The sesquilinear form $a(\cdot,\cdot)$ is bounded: there exists a constant $\Const{bdd}$ independent of wavenumber $k$
such that
    \begin{align}\label{eq:sesqui-bound}
    |a(v_1,v_2)|\leq \Const{bdd}\|v_1\|_{V}\|v_2\|_V, \quad\forall v_1,v_2\in V.
    \end{align}
\item[2.] The G\aa rding's inequality holds
\begin{align}\label{eq:garding}
\Re a(v,v)+2k^2\normL{v}{\Omega}^2\geq \|v\|_V^2,\quad\forall v\in V.
\end{align}
\end{itemize}
\end{theorem}
The existence and uniqueness of the weak solution of \eqref{var_form} can be found in \cite[Proposition 8.1.3]{melenk1995generalized}. 
\begin{comment}
Furthermore, there exists some constant, $\Co$, that may depend on the wavenumber $k$, such that the unique solution $u\in V$ to Problem \eqref{var_form} fulfills
\begin{align}\label{res:apriori}
\norm{u}_{V}\leq \Co \normL{f}{\Omega}.
\end{align}
\end{comment}
The adjoint problem to \eqref{var_form} is frequently utilized in the analysis of indefinite problems satisfying the G\aa rding's inequality \eqref{eq:garding}, which is defined as given $w\in V'$, find $z\in V$ such that
\begin{align}\label{eqn:weakformDual}
a^*(z, v)=(w,v)_{\Omega}, \quad\forall v\in V,
\end{align}
where $a^*(\cdot,\cdot)$ is defined by (see, e.g., \cite{MR1742312})
$a^*(w, v)=\overline{a(v,w)}$ for all $ w,v\in V$.
To discretize problem \eqref{hel_cont}, we first decompose the domain $\Omega$. Let $\mathcal{T}_h$ be a uniform triangulation of $\Omega$ with a mesh size $h$. 
For simplicity, let $V_h$ be the $H^1$-conforming piecewise linear finite element space
\[
V_h:=\{v\in H^{1}(\Omega): v|_{T}\in \mathcal{P}_{1}(T),\;\forall T\in \mathcal{T}_h\},
\]
where $\mathcal{P}_1(T)$ denotes the space of linear polynomials on the fine element $T\in \mathcal{T}_h$. Let $\Pi_h: V\to V_h$ be the standard nodal interpolation operator. The $H^1$-conforming Galerkin approximation to problem \eqref{var_form} is to find $u_h\in V_h$ such that 
\begin{equation}\label{eq:discrete-global}
    a(u_h,v_h)=F(v_h), \quad\forall v_h\in V_h.
\end{equation}
The generalization to high order method is straightforward. Throughout, we assume the approximation property for the FEM space $V_h$: for all $v\in H^2(\Omega)$, there holds
\begin{align}\label{approxi}
   \inf_{v_h\in V_h} \left(\|v-v_h\|_{L^2(\Omega)}+h\|\nabla (v- v_h)\|_{L^2(\Omega)}\right)\lesssim h^2|v|_{H^2(\Omega)}.
\end{align}
Using property \eqref{approxi}, we can establish the following stability of the discrete problem \cite{MR3232452}. 
\begin{theorem}[Discrete stability]
\label{thm_1}
Let $\Omega$ be a convex polygon/polyhedron domain. Then there exists a constant $C_0$ sufficiently small that is independent of $k$ and $h$, such that if the triangulation mesh satisfies %\ednote{Shihua: Consider to use the update-to-date results on preasymptotic error analysis, such as $h^2k^3\le C_0$. If yes, we need to rewrite the remark. }
\begin{align}\label{eq:fine-mesh-cond}
h^2k^{3}\leq C_0,
\end{align}
%then the following discrete inf-sup condition holds
%\begin{equation*}
 %   \inf_{v_h\in V_h\backslash\{0\}}\sup_{w_h\in V_h\backslash\{0\}}\frac{ |a(v_h,w_h)|}{\|v_h\|_{V}\|w_h\|_{V}}
  %  \geq \frac{1}{\Const{stab}k}
%\end{equation*}
%with the constant $\Const{stab}$ independent on $h$ and $k$. 
then \eqref{eq:discrete-global} is well-posed, and its solution $u_h$ satisfies
\begin{equation*}
    \|u_{h}\|_{V} \leq\Const{stab}  (\|f\|_{L^2(\Omega)}+\|g\|_{L^2(\partial\Omega)}). 
\end{equation*}
Here, the constant $\Const{stab}$ is independent of $h$ and $k$. 
%where $u$ is the exact .
\end{theorem}

\begin{remark}\label{rm:mesh-resolution}
%The mesh resolution condition \eqref{eq:fine-mesh-cond} guarantees that the numerical solution $u_h$ achieves a higher accuracy as the wavenumber $k$ increases. 
%It can be relaxed if $H^1$-conforming piecewise $p$ order finite element space is used with $p$ large or mere existence of numerical solutions is pursued. 
The existence of numerical solutions with large wavenumber $k$ was studied in the seminal work of Ihlenburg and Babu\v{s}ka \cite{ihlenburg1995dispersion, ihlenburg1995finite}. Under the assumption that the stability of the Helmholtz problem is independent of $k$, the Galerkin method using the $h$-version of FEM is quasi-optimal when $k^{p+1}h^p$ is sufficiently small \cite{melenk1995generalized, sauter2006refined, melenk2010convergence, MR2812565}. For the $hp$-version (where accuracy is increased by decreasing $h$ and increasing $p$), the Galerkin method is quasi-optimal if $hk/p$ is sufficiently small and $p$ grows logarithmically with $k$ by \cite{melenk2010convergence, MR2812565}. %Under a weaker condition that $k^{2p+1}h^{2p}$ is small enough, the numerical solution exists and the numerical error is bounded by the data uniformly with respect  to $k$ \cite{du2015preasymptotic}. 
In all these cases with quasi-uniform meshes, the magnitude of the dimensions of the piecewise polynomial function spaces is much larger than $k^{d}$.
\end{remark}
\begin{remark}\label{final:1111}
Theorem \ref{thm_1} is valid if $a(\cdot,\cdot)$ in \eqref{eq:discrete-global} is replaced by $a^*(\cdot,\cdot)$. 
\end{remark}
\section{The hybrid Schwarz method}\label{sec:hybrid-2}
Now we present the basic assumption for the domain decomposition and the notation for the Schwarz preconditioners. Let 
$\mathcal{T}_H : = \{K_i~|~i=1,2,\cdots, N\}
$ be partition of the domain $\Omega$ into a quasi-uniform mesh with a diameter $H$ such that $\mathcal{T}_h$ is its subdivision, where $N$ is the number of the coarse elements (or the subdomains). Let $\Omega_{i}$ be the smallest subdomain associated with $K_{i}\in \mathcal{T}_H$, such that
$\bigcup\left\{T\in\mathcal{T}_h:\text{dist}(T,K_{i})\leq\delta_{i}\right\}\subset \Omega_i$ and the shape of $\Omega_i$ is similar to that of $K_i$.
Denote by $\delta\coloneqq\min\{\text{dist}(\partial K_i\backslash\partial\Omega,\partial \Omega_i\backslash\partial\Omega):{i}=1,\cdots,N\}$ the smallest value. Then $\{\Omega_{i}\}_{i=1}^N$ is an overlapping partition of $\Omega$, i.e., 
$\overline{\Omega}=\bigcup_{i=1}^N\overline{\Omega}_{i}$. We denote by $d_i\coloneqq\op{diam}(\Omega_i)$ the diameter of $\Omega_i$ and $\bar{d}\coloneqq \max\{d_i:i=1,\cdots,N\}$ the largest diameter of the subdomains.
%Throughout, we assume $\delta$ is small such that
%\begin{align}\label{eq:c-diameter}
%d_i\simeq H, \quad\forall i=1,\cdots,N.
%\end{align}
We make the {\it finite-overlap assumption}: there exists a finite number $\Lambda>1$ independent of $N$ such that
\begin{align}\label{eq:finite-overlap}
\Lambda=\max\left\{\#\Lambda_{i}:K_{i}\in\mathcal{T}_H\right\}, \quad \text{ with }\Lambda_{i}=\{j:\bar{\Omega}_{i}\cap\bar{\Omega}_{j}\neq \emptyset\}.
\end{align}

The diameter of each subdomain $d_i$ is required to satisfy the so-called scale resolution assumption. 
\begin{assumption}[Scale Resolution Assumption]\label{ass:resolution}
The diameter $d_i$ is sufficiently small so that 
\begin{align*}
\max_{i=1,2,\ldots,N}\{\Cpoinn{\Omega_i}{1/2}\}d_ik<1,
\end{align*}
with the Poincar\'{e} constant $\Cpoin{\Omega_i}: =\max_{v\in H^1_0(\Omega_i)}
\frac{\|v\|_{L^2(\Omega_i)}^2}{\|d_i\nabla v\|_{L^2(\Omega_i)}^2}$ depending on the shape of $\Omega_i$ but not on the diameter $d_i$.
\end{assumption}

Throughout, we denote 
\begin{equation}\label{eq:cest}
\Ce:=\Big(1-\max_{i=1,2,\ldots,N}\{\Cpoin{\Omega_i}\}(d_ik)^2\Big)^{-1}.
\end{equation}
Note that for most regular domains, the Friedrichs' inequality implies that $\Cpoin{\Omega_i}:= 1/\pi^2$. Hence, Assumption \ref{ass:resolution} is equivalent to $d_i<  \pi k^{-1}$.
Next, we introduce a discrete nonnegative-valued partition of unity $\{\chi_{i}\}_{i=1}^N\subset V_h$ subordinate to the cover $\{\Omega_{i}\}_{i=1}^N$, which is assumed to satisfy for some $C_{\op{G}}$ independent of $(h,k,\bar{d},\delta)$,
\begin{equation}\label{eq:pum-property}
\left\{
\begin{aligned}
\sum\limits_{i=1}^N\chi_{i}&=1 \text{ in } \bar{\Omega},\quad\text{and}\quad {\text{supp}(\chi_{i})}\subset\bar{\Omega}_{i},\\
 0\leq\chi_{i}&\leq 1,\quad\qquad\text{and}\quad 
\|\nabla \chi_{i}\|_{L^{\infty}(\bar{\Omega}_{i})}\leq {\delta_{i}}^{-1}C_{\op{G}}.
\end{aligned}
\right. 
\end{equation}
It follows immediately from \eqref{eq:finite-overlap} that, for $v_i\in H^1(\Omega_i)$,
\begin{equation}\label{eq:sumpou-ieq}
\left\|\sum_{i=1}^N \chi_i v_i\right\|_{L^2(\Omega)}^2 \le \Lambda \sum_{i=1}^N \left\| \chi_i v_i\right\|_{L^2(\Omega_i)}^2,\quad \text{and}\quad \left\|\sum_{i=1}^N\nabla( \chi_i v_i)\right\|_{L^2(\Omega)}^2 \le \Lambda \sum_{i=1}^N \left\| \nabla(\chi_i v_i)\right\|_{L^2(\Omega_i)}^2.    
\end{equation}

\begin{comment}
Then a slight modification of {\cite[Theorem 2.1]{melenk1996partition}} leads to the following result,
\begin{lemma}
\label{thm:pum}
Let $v_i\in V_{h,i}$ for $i=1,\cdots,N$, and let
\begin{align*}
v_h:=\sum_{i=1}^N \chi_i v_i\in V_h.
\end{align*}
Then there holds
\begin{align*}
\|v_h\|_{L^2(\Omega)}&\leq \sqrt{\Cov}C_{\infty}\left(\sum_{i=1}^N\|v_i\|_{L^2(\Omega_i)}^2\right)^{1/2},\\
\|\nabla v_h\|_{L^2(\Omega)}&\leq \sqrt{2\Cov}\left(\sum_{i=1}^NC_{\text{G}}^2\delta_i^{-2}\|v_i\|_{L^2(\Omega_i)}^2
+\|\chi_i\nabla v_i\|_{L^2(\Omega_i)}^2\right)^{1/2}.%\label{eq:glo-err-pum}
\end{align*}
\end{lemma}
\end{comment}

\subsection{RAS-imp preconditioner}
Now we formulate the one-level Schwarz domain decomposition method in an operator form.
First we define the linear operator ${A}:V_{h}\to V_{h}'$ by
\begin{align}\label{eq:Ah-linear}
\langle {A}v_h,w_h\rangle:=a(v_h,w_h),\quad \forall v_h, w_h \in V_{h},
\end{align}
where $\langle \cdot,\cdot\rangle$ denotes the duality pairing. 
Then the global linear system can be written as 
\begin{equation}
    \label{lin_form}
   A u_h=F_h,
\end{equation}
where $F_h\in V_h'$ is defined as $\langle F_h,v_h\rangle := F(v_h)$~ for all $v_h\in V_h$.

Next, we consider the local problems with impedance boundary conditions. Let the local function spaces be
$V_{h,i}:=\{v_h|_{\overline{\Omega}_{i}}: v_h\in V_h\}$ and define the local operator $A_i :V_{h,i}\to V_{h,i}'$ by 
\begin{align*}
\langle {A}_{i}u_i,v_i\rangle:=a_{i}(u_i,v_i), \forall u_i, v_i \in V_{h,i},
\end{align*}
where the local sesquilinear form $a_{i}(\cdot,\cdot)$ on the local space $V_{h,i} \times V_{h,i}$ is given by
\begin{equation*}
\label{local_ses}
    a_{i}(u_i,v_i) 
    := \int_{\Omega_i}\nabla u_i\cdot\nabla\overline{ v}_i\mathrm{d}x
    -k^2\int_{\Omega_i} u_i\overline{v}_i\;\mathrm{d}x
    + \mathrm{i}k\int_{\partial \Omega_i} u_i\overline{ v}_i \;\mathrm{d}s,\quad
    \forall u_i, v_i\in V_{h,i}.
\end{equation*}
The sesquilinear form $a_i(\cdot,\cdot)$ is bounded: there exists a constant $\Const{bdd}$ independent of wavenumber $k$
such that
    \begin{align}\label{eq:sesqui-bound-local}
    |a_i(v_1,v_2)|\leq \Const{bdd}(1+(kd_i)^{-1})\|v_1\|_{V(\Omega_i)}\|v_2\|_{V(\Omega_i)}, \quad\forall v_1,v_2\in V(\Omega_i).
    \end{align}
\begin{comment}
\begin{lemma}
If each $\Omega_{i}$ is starshaped with respect to a ball uniformly in $i$, then there exists a mesh threshold function $\bar{h}(k,r)$ such that when $h\leq \bar{h}(k,r)$, ${A}_{i}:V_{h,i}\to V_{h,i}'$ is invertible, and 
\begin{align*}
\|{A}_{i}^{-1}\|\leq 1+Hk.
\end{align*}
\end{lemma}
\end{comment}
Then we define prolongations and restrictions to link local and global problems. Note that the functions in $V_h$ and $V_{h,i}$ are uniquely determined by their nodal values at the nodes $\{x_j: j\in \mathcal{I}\}$ in $\overline{\Omega}$ and $\{x_j: j\in \mathcal{I}_i\}$ in $\overline{\Omega}_i$, respectively, using the index sets $\mathcal{I}$ and $\mathcal{I}_i$. We define the prolongation ${R}_i^T: V_{h,i} \to V_h$  by
\begin{align*}
R_i^Tv_i(x_j)=\left\{
\begin{aligned}
v_i(x_j), \quad j\in \mathcal{I}_i,\\
0,\quad \text{otherwise}.
\end{aligned}\right.
\end{align*}
Note that the prolongation ${R}_i^T$ is defined nodewise and ${R}_i^Tv_i $ is an $H^1$-conforming finite element approximation to the zero extension of $v_{i}\in V_{h,i}$. (The zero extension is not in $H^1(\Omega)$ in general). We define $R_i: V'_{h}\to V'_{h,i}$ by duality. To make the prolongation more stable, we employ the weighted prolongation $\tilde{R}_i^T:V_{h,i}\to V_h$ by the functions of partition of unity:
\begin{equation*}
\tilde{R}_i^T v_{i}:  = {R}_i^T \Pi_{h,i}(\chi_i v_i),
\end{equation*}
where $\Pi_{h,i}$ is the nodal interpolation of $V_{h,i}$. The nodal interpolation operator $\Pi_{h,i}$ has the stability property (see, e.g., \cite{MR520174}),
\begin{align}\label{eq:nodal-interp}
\left\|(I-\Pi_{h,i})v_i\right\|_{L^2(\Omega_i)}
+h\left\|\nabla (I-\Pi_{h,i}) v_i\right\|_{L^2(\Omega_i)}
\leq \Const{I}h\left\|\nabla v_i\right\|_{L^2(\Omega_i)},\quad\forall v_i\in V_{h,i}.
\end{align}
Here, $\Const{I}$ depends only on the shape regularity of $\mathcal{T}_h$. 
Note that 
\begin{align}\label{eq:pou}
\sum_{i=1}^N \tilde{R}_i^T(v_h|_{\Omega_i})=v_h, \quad\forall v_h\in V_h.
\end{align}
The restricted additive Schwarz preconditioner with local impedance boundary condition (RAS-imp) is the operator $B_{\op{RAS-imp}}: V_{h}'\to V_h$ defined by
\begin{equation*}
B_{\op{RAS-imp}} = \sum_{i=1}^N \tilde{R}_i^T  A_i^{-1}R_i.
\end{equation*}
The preconditioned operator $B_{\op{RAS-imp}} A :V_h\to V_h$ can be written as
\begin{equation*}
B_{\op{RAS-imp}}A = \sum_{i=1}^N  \tilde{R}_i^T Q_i,
\end{equation*}
where the local Schwarz operator $Q_i:V_h\to V_{h,i}$ is defined by
\begin{align*}
a_i(Q_i v, w_i) = a(v,R_i^T w_i),\quad \forall v\in V_h\text{ and } w_i \in V_{h,i}.
\end{align*}
\begin{lemma}[Local-global approximation]
\label{thm:pum-2}
Let $v_i\in V_{h,i}$~~ for~~ $i=1,2,\cdots,N$, and let
$v_h:=\sum_{i=1}^N \tilde{R}_i^T v_i\in V_h$.
Then there exists $\Const{I}$ independent of $(k,h,H)$ such that
\begin{align}
\|v_h\|_{V(\Omega)}\leq \sqrt{\Lambda}(1+\Const{I})\left(\sum_{i=1}^N \|\chi_iv_i\|_{V(\Omega_i)}^2
\right)^{1/2}.\label{eq:glo-err-pum3}
\end{align}
\begin{comment}
and 
\begin{align}\label{eq:glo-err-pum2}
\|v_h\|_{V}\leq \sqrt{\Lambda}(1+\Const{I})\left(C_{\infty}+k^{-1}
\delta^{-1}C_{\op{G}}\right)\left(\sum_{i=1}^N \|v_i\|_{V(\Omega_i)}^2
\right)^{1/2}.
%\leq \sqrt{\Lambda}(1+\Const{I})\left(\sum_{i=1}^N \|\chi_iv_i\|_{V(\Omega_i)}^2\right)^{1/2}.
\end{align}
\end{comment}
\end{lemma}
\begin{proof}
A combination of the triangle inequality, the stability of  $\Pi_{h,i}$ \eqref{eq:nodal-interp} and \eqref{eq:sumpou-ieq}, together with the finite-overlap assumption \eqref{eq:finite-overlap}, leads to
\begin{align}
\|v_h\|_{L^2(\Omega)}&\leq \sqrt{\Lambda}\left(\sum_{i=1}^N \|\chi_iv_i\|_{L^2(\Omega_i)}^2
+\frac{1}{2}\Const{I}^2h^2\left\|\nabla (\chi_i v_i)\right\|_{L^2(\Omega_i)}^2\right)^{1/2},\nonumber\\
\|\nabla v_h\|_{L^2(\Omega)}&\leq \sqrt{\Lambda(1+\frac{1}{2}\Const{I}^2)}\left(\sum_{i=1}^N
\left\|\nabla (\chi_i v_i)\right\|_{L^2(\Omega_i)}^2\right)^{1/2}.\label{eq:glo-err-pum2_2}
\end{align}
By the definition of $\Vert\cdot\Vert_{V}$, we get
\begin{equation}
    \begin{aligned}
        \|v_h\|_{V}^2&=\Vert\nabla v_h\Vert_{L^2(\Omega)}^2+k^2\Vert v_h\Vert_{L^2(\Omega)}^2\\
        &\leq\sum_{i=1}^Nk^2\Lambda \|\chi_iv_i\|_{L^2(\Omega_i)}^2
+\Lambda\left(1+\frac{1}{2}\Const{I}^2+\frac{1}{2}k^2h^2\Const{I}^2\right)\left\|\nabla (\chi_i v_i)\right\|_{L^2(\Omega_i)}^2\\
        &\leq \Lambda(1+\Const{I}^2)\sum_{i=1}^N \|\chi_iv_i\|_{V(\Omega_i)}^2,
    \end{aligned}
\end{equation}
where the last inequality follows from the fact that $kh<1$. This yields the desired assertion. 
\begin{comment}
Since 
\begin{align*}
    \|\chi_iv_i\|_{V(\Omega_i)}&\leq\left(k^2\left(C_{\infty}^2+k^{-2}{\delta_{i}}^{-2}C_{\op{G}}^2\right)\|v_i\|_{L^2(\Omega_i)}^2+C_{\infty}^2\|\nabla v_i\|_{L^2(\Omega_i)}^2\right)^{1/2}\\
    &\leq\left(C_{\infty}^2+k^{-2}{\delta_{i}}^{-2}C_{\op{G}}^2\right)^{1/2}\|v_i\|_{V(\Omega_i)},
\end{align*}
this yields the second assertion.
\end{comment}
\end{proof}
\subsection{The EMs-HS preconditioner}
The hybrid preconditioner is a multiplicative assembly of the one level RAS-imp preconditioner $B_{\op{RAS-imp}}$ with a coarse solver, which is determined by a properly defined coarse space $V_{h,0}\subset V_h$ (see Section \ref{sec:coarse-space} for details). Assuming a natural prolongation $R_0^T:V_{h,0} \to V_h$, we define the coarse problem 
\begin{align*}
A_0 := R_0 A R_0^T : V_{h,0}\to V_{h,0}'.
\end{align*}
where $R_0 :V_{h}'\to V_{h,0}'$ is defined by duality. Then the EMs-HS preconditioner is defined by 
\begin{equation*}
B_{\op{hybrid}} : =  R_0^TA_0^{-1} R_0(I - A B_{\op{RAS-imp}}) + B_{\op{RAS-imp}}.
\end{equation*}
Note that the RAS-imp preconditioner $B_{\op{RAS-imp}}$ is applied only once in the computation of $B_{\op{hybrid}}$. More precisely, given a residual equation $Ae = r$, one first computes the one-level correction by $e^{(1)} = B_{\op{RAS-imp}}r$, then calculates the new residual $r^{(1)} = r - A e^{(1)}$ and solves for the coarse correction $e^{(2)} = R_0^TA_0^{-1} R_0r^{(1)}$. Then the final correction is  
$B_{\op{hybrid}} r = e^{(1)}+e^{(2)}$.

The  preconditioned operator can be written as 
\begin{equation*}
\begin{aligned}
B_{\op{hybrid}}A &= R_0^TA_0^{-1} R_0A(I -  B_{\op{RAS-imp}}A) + B_{\op{RAS-imp}}A\\
%&= R_0^TQ_0(I - B_{RAS-imp}A) + B_{RAS-imp}A\\
\end{aligned}.
\end{equation*}
Analogously, we define the coarse Schwarz operator $Q_0 = A_0^{-1} R_0A:V_h\to V_{h,0} $. Using the  preconditioner $B_{\op{hybrid}}$, we obtain a fixed-point iteration for \eqref{lin_form}
\begin{align*}
u^{(n+1)}_h = u^{(n)}_h +B_{\op{hybrid}} (F_h - A u^{(n)}_h), 
\end{align*}
with its associate error transfer operator being
\begin{equation*}
\begin{aligned}
I-B_{\op{hybrid}}A  &= \left(I-R_0^TQ_0\right)\left(I - B_{\op{RAS-imp}}A\right) 
 = \Big(I-R_0^TQ_0\Big)\left(I - \sum_{i=1}^N\tilde{R}^T_i Q_i\right).
\end{aligned}
\end{equation*}
Furthermore, \eqref{eq:pou} implies
\begin{equation*}
\left(I-B_{\op{hybrid}}A \right)v_h 
= \left(I-R_0^TQ_0\right)\sum_{i=1}^N \tilde{R}^T_i  \left(v_h|_{\Omega_i}- Q_i v_h\right) \quad\forall v_h\in V_h.
\end{equation*}
In the analysis, we define, for any $v_h\in V_h$, 
\begin{align}
v_h^{\partial}:=\sum_{i=1}^N \tilde{R}^T_{i}\left(v_h|_{\Omega_i}- Q_i v_h\right)
\quad\mbox{and then } \left(I-B_{\op{hybrid}}A\right) v_h=\left(I-R_0^TQ_0\right)v_h^{\partial}. \label{eq:proj-edge-err}
\end{align}
The main objective is to show that the operator $I-B_{\op{hybrid}}A $ is contractive under an explicit requirement on $(h,k,\bar{d},\delta)$.

We end this section with an auxiliary result on the discrepancy between the local Schwarz projection $Q_i v_h$ and the restriction  $v_h|_{\Omega_i}$.

\begin{lemma}\label{lem:decomp}
Let the fine grid mesh size $h$ and the diameter $d_i$ of each subdomain $\Omega_i$ satisfy \eqref{eq:fine-mesh-cond} and Assumption \ref{ass:resolution}, respectively. Given $v_h\in V_h$, let $v_h^{\partial}$ be defined in \eqref{eq:proj-edge-err}.  Then with $\sigma:=(k\bar{d})^{-3/2}(kh)^{-1/2}$, there exists $C_{\partial}$ independent of $(h,k,\bar{d},\delta)$ such that
\begin{align}\label{eq:vhpar}
\left(\sum_{i=1}^N\left\|Q_i v_h-v_h|_{\Omega_i}\right\|_{V(\Omega_i)}^2\right)^{1/2}\leq C_{\partial}\sigma \|v_h\|_V.
\end{align}
\end{lemma}
\begin{proof}
By \cite[Lemma 4.4]{MR4890903},  for each $i=1,\cdots,N$, there exists $\sigma_i:=(kd_i)^{-3/2}(kh)^{-1/2}$ such that
\begin{align}\label{eq:local-harmonic}
\left\|Q_i v_h-v_h|_{\Omega_i}\right\|_{V(\Omega_i)}\lesssim \sigma_i\|v_h\|_{V(\widetilde{\Omega}_{i})},
\end{align}
where $\widetilde{\Omega}_{i}$ satisfies $
\overline{\widetilde{\Omega}}_{i}:=\cup\{\overline{T}\in \mathcal{T}_h: \overline{T}\cap\Omega_i\neq \emptyset\}.$
%This, Theorem \ref{thm:pum-2} and \eqref{eq:proj-edge-err} lead to 
%\begin{align*}
%\|v_h^{\partial}\|_V^2
%&\lesssim(1+\delta_i^{-2}k^{-2})\sum_{i=1}^N \left\|Q_i v_h-v_h|_{\Omega_i}\right\|_{V(\Omega_i)}^2\lesssim\sigma^2
%\sum_{i=1}^N \|v_h\|_{V(\widetilde{\Omega}_{i})}^2.
%\end{align*}
Note that the finite-overlap assumption \eqref{eq:finite-overlap} ensures that $\{\widetilde{\Omega}_i\}_{i=1}^N$ satisfies the same assumption. Hence, taking square summation over both sides of \eqref{eq:local-harmonic}, this yields the desired assertion.
\end{proof}
%========================================================
\section{Coarse space}\label{sec:coarse-space}
To define the coarse space, we employ the local Helmholtz harmonic space $V_i^\partial$ with Dirichlet boundary data
\begin{align}\label{eq:local-space}
V_i^{\partial}=\{v_h\in V_{h,i}: a_i(v_h,w_h)=0\; \forall w_h\in V_{h,i}^0\}.
\end{align}
Here, $V_{h,i}^0\subset V_{h,i}$ denotes the space consisting of functions vanishing on the boundary $\partial\Omega_i$, namely, 
\[
V_{h,i}^0:=\{v_h\in V_{h,i}:\;v_h=0\text{ on }\partial\Omega_i\}.
\]
Utilizing the weighted prolongation $\tilde{R}_i^T: V_{h,i}\to V_h$ and the local spaces $V_i^{\partial}$, we define a global space
\begin{align}\label{eq:glo-edge-space}
V^{\partial}\coloneqq\mathrm{span}\left\{\sum_{i=1}^N \tilde{R}_i^T v_i^{\partial}: v_i^{\partial}\in V_i^{\partial}\right\}.
\end{align}
Next, we construct the coarse space $V_{h,0}$ following \cite{MR4860976}, which approximates the global space $V^{\partial}$ \eqref{eq:glo-edge-space} and also serves as a good ansatz space for the adjoint problem \eqref{eqn:weakformDual}. 
%=======================================
\subsection{Hierarchical subspace splitting over $I=:[0,1]$}\label{sec:wavelets}
%=============================
First, we introduce the hierarchical bases on the unit interval $I:=[0,1]$, which facilitate hierarchically splitting the space $L^2(I)$ \cite{MR1162107}. Let $\ell\in \mathbb{N}$ be the level parameter and $h_{\ell}:=2^{-\ell}$ be the mesh size, respectively. The grid points on level $\ell$ are given by
\[
x_{\ell,j}=j\times h_{\ell},\quad 0\leq j\leq 2^{\ell}.
\]
%Let the scaling function and the mother wavelet be given by
%By means of dilation and translation, the mother wavelet $\psi(x)$ can result in orthogonal decomposition of the space $L^2(I)$. To this end,
We define the bases on level $\ell$ by
\begin{equation*}
\psi_{\ell,j}(x)=
\left\{
\begin{aligned}
&1-|x/h_{\ell}-j|, &&\text{ if }  x\in [(j-1)h_{\ell},(j+1)h_{\ell}]\cap [0,1],\\
&0, &&\text{ otherwise},
\end{aligned}
\right.
%\qquad
%\psi(x)=
%\left\{
%\begin{aligned}
%&1, &&\text{ if } 0\leq x\leq 1/2,\\
%&-1, &&\text{ if }1/2< x\leq 1,\\
%&0, &&\text{ otherwise}.
%\end{aligned}
%\right.
\end{equation*}
and the index set on each level $\ell$ by
\begin{equation*}
B_{\ell}:=\Bigg\{
j\in\mathbb{N}\Bigg|
\begin{aligned}
&j=1,\cdots,2^{\ell}-1, j \,\rm{ is\, odd }, &&\rm{if }\,\ell>0\\
&j=0,1,&&\rm{ if }\,\ell=0
\end{aligned}
\Bigg\}.
\end{equation*}
%Note that the set $B_{\ell}$ has the nesting property, namely, $B_{\ell-1}\subset B_{\ell}$ for all $1<\ell\in\mathbb{N}_{+}$.
%\begin{align*}
%\psi_{\ell,j}(x):=2^{\ell/2}\psi(2^{\ell}x-j) \quad \text{ for all }\quad 0\leq j\leq 2^{\ell-1}-1.
%\end{align*}
The subspace of level $\ell$ is
$W_{\ell}:=\text{span}\{\psi_{\ell,j}: j\in B_{\ell}\}$.
%The nesting property of the set $B_{\ell}$ leads to the nesting property of the subspace $W_{\ell}$, namely, the subspace $W_{\ell}\subset W_{\ell'}$ for any two different levels $\ell<\ell'$.
We denote by $V_{\ell}$ the subspace in $L^2(I)$ up to level $\ell$, which is the direct sum of subspaces $V_{\ell}:=\oplus_{m\leq\ell}W_{m}$.
This gives a hierarchical structure of the subspace $V_{\ell}$:
$V_{0}\subset V_{1}\subset \cdots\subset V_{\ell}\subset V_{\ell+1}\cdots$.
Furthermore, the following hierarchical decomposition of the space $L^2(I)$ holds
\[
L^2(I)=\lim_{\ell\to\infty}\oplus_{m\leq\ell}W_{m}.
\]
The hierarchical decomposition of the space $L^2(I^{d-1})$ for $d=2,3$ can be defined by means of tensor product, which is denoted as $V_{\ell}^{\otimes^{d-1}}$. We will use the subspace $V_{\ell}^{\otimes^{d-1}}$ to approximate the restriction of the exact solution $u$ on the coarse skeleton $\partial\mathcal{T}^H:=\cup_{T\in \mathcal{T}^H}\partial T$.

Next, we present the approximation properties of the hierarchical space $V_{\ell}$, which can be derived from standard $L^2$-projection error \cite{MR520174} combining with the interpolation method.
\begin{proposition}[Approximation properties of the hierarchical space $V_{\ell}$]\label{prop:approx-wavelets}
Let $d=2,3$, $s\in (0,1]$ and let $\mathcal{I}_{\ell}: C(I^{d-1})\to V_{\ell}^{\otimes^{d-1}}$ be $L^2$-projection for each level $\ell\geq 0$, then there holds,
\begin{align}\label{eq:approx-wavelets}
%P_{\ell+1} v&=P_{\ell}v+\sum\limits_{j=0}^{2^{\ell}-1}(v,\psi_{\ell+1,j})_{I}\psi_{\ell+1,j}&\text{for all }v\in L^2(I)\\
\|v-\mathcal{I}_{\ell}v\|_{L^2(I^{d-1})}&\lesssim 2^{-s\ell}|v|_{H^s(I^{d-1})}\quad\text{for all }v\in H^s(I^{d-1}).
\end{align}
Here, $|\cdot|_{H^s(I^{d-1})}$ denotes the Gagliardo seminorm in the fractional Sobolev space (or Slobodeskii space) $H^s(I^{d-1})$, given by
\begin{align*}
|v|^2_{H^s(I^{d-1})}:=\int_{I^{d-1}}\int_{I^{d-1}}\frac{|v(x)-v(y)|^2}{|x-y|^{d-1+2s}}\mathrm{d}x\mathrm{d}y.
\end{align*}
The corresponding full norm is denoted as $\|\cdot\|_{H^s(I^{d-1})}$.
\end{proposition}
\subsection{Coarse space}
%==============================
Now we present the construction of the coarse space $V_{h,0}$. The key is to
find local multiscale bases in each coarse subdomain $\Omega_i$, having certain approximation properties to $V_i^{\partial}$. The global multiscale bases functions, obtained from the local ones by the weighted prolongation $\tilde{\mathcal{R}}_i^T$, inherit these approximation properties. We will illustrate it on the case with $d = 2$ for the sake of simplicity, even though our methodology can be extended to $d = 3$ naturally.

First we define the vector space over the boundary of each subdomain $\partial\Omega_i$. Fix the level parameter $\ell\in \mathbb{N}$, and let $\Gamma_{i}^j$ with $j=1,\cdots,m_i$ be a uniformly shape-regular decomposition of $\partial\Omega_i$ into affine edges/faces with no mutual intersection and with $m_i$ uniformly bounded independent of $h$ and $H$, i.e., $\cup_{j=1}^{m_i}\overline{\Gamma_{i}^j}=\partial\Omega_i$ and $\Gamma_{i}^j\cap \Gamma_{i}^{j'}=\emptyset$ if $j\neq j'$. 
Furthermore, we defined $E_{i,\ell}:=\{\psi_h\in C(\gamma_i):\psi_h|_{\Gamma_{i}^j}\in V_{\ell}(\Gamma_{i}^j)\}$, where $V_{\ell}(\Gamma_{i}^j)$ is the vector space spanned by hierarchical bases up to level $\ell$ on coarse edge $\Gamma_{i}^{j}$. We assume that $V_h(\Gamma_{i}^j)\coloneqq \{v|_{\Gamma_i^j}:v\in V_h\}$ is a refinement of $V_{\ell}(\Gamma_{i}^j)$. Let $\{\psi_{i,\ell}^k\}_{k=1}^{n_{i,\ell}}$ and 
$\{x_{i,\ell}^k\}_{k=1}^{n_{i,\ell}}$ be the nodal bases and the associated nodal points for $E_{i,\ell}$. Then we can represent the local edge space $E_{i,\ell}$ by  
\begin{align}\label{eq:local-edge}
E_{i,\ell} := \text{span} \big\{\psi_{i,\ell}^k : 1 \leq k \leq  n_{i,\ell}\big\}.
\end{align}
We depict in Figure \ref{fig:grid-points-local-new} the grid points $\{x_{i,\ell}^k\}_{k=1}^{n_{i,\ell}}$ of the hierarchical bases over $\partial\Omega_i$ for level parameter $\ell\in\{0,1,2\}$. 

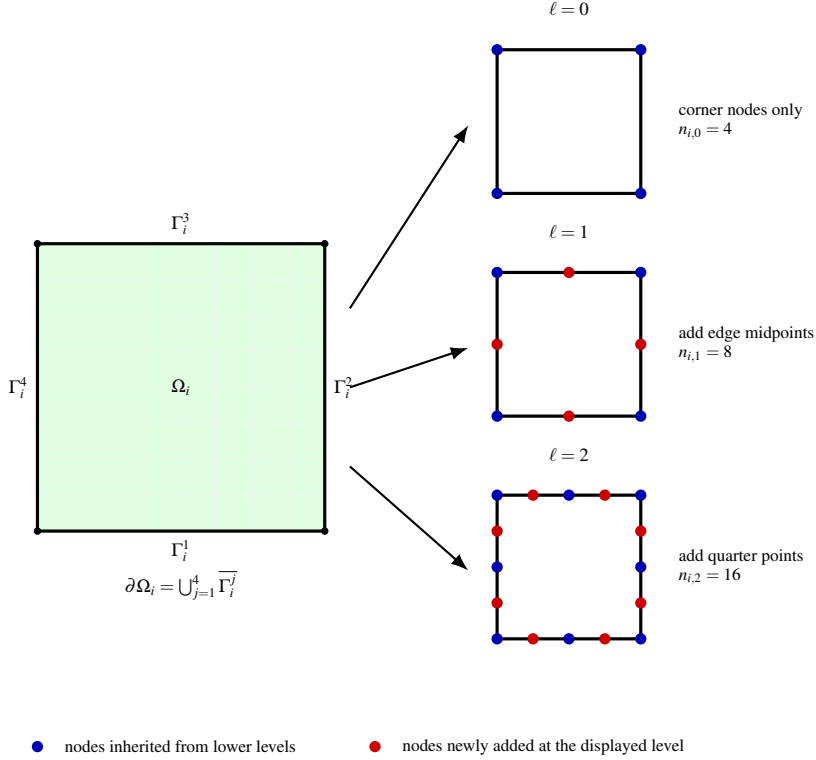
\begin{figure}[htbp]
\centering
\begin{tikzpicture}[
    x=0.95cm,
    y=0.95cm,
    >=Latex,
    every node/.style={font=\small}
]

% -------------------------------------------------
% styles
% -------------------------------------------------
\tikzset{
  oldpt/.style={circle,fill=blue!70!black,inner sep=1.5pt},
  newpt/.style={circle,fill=red!80!black,inner sep=1.5pt},
  edgebox/.style={draw=black,very thick},
  myarrow/.style={->,thick}
}

% -------------------------------------------------
% Left: coarse subdomain Omega_i
% -------------------------------------------------
\begin{scope}[shift={(0,0)}]
    % fill and mesh
    \fill[green!12] (0,0) rectangle (4,4);
    \draw[step=0.5,gray!35,thin] (0,0) grid (4,4);
    \draw[edgebox] (0,0) rectangle (4,4);

    % vertices
    \foreach \P in {(0,0),(4,0),(4,4),(0,4)}
        \fill \P circle (1.4pt);

    % labels
    \node at (2,2) {$\Omega_i$};

    \node[below] at (2,0) {$\Gamma_i^1$};
    \node[right] at (4,2) {$\Gamma_i^2$};
    \node[above] at (2,4) {$\Gamma_i^3$};
    \node[left]  at (0,2) {$\Gamma_i^4$};

    % small label for boundary
    \node at (2,-0.75) {$\partial\Omega_i=\bigcup_{j=1}^4 \overline{\Gamma_i^j}$};
\end{scope}

% arrows from the left panel to the level panels
\draw[myarrow] (4.35,3.10) -- (6.00,5.65);
\draw[myarrow] (4.35,2.00) -- (6.00,2.55);
\draw[myarrow] (4.35,0.90) -- (6.00,-0.55);

% -------------------------------------------------
% Right top: ell = 0
% -------------------------------------------------
\begin{scope}[shift={(6.4,4.7)}]
    \draw[edgebox] (0,0) rectangle (2,2);
    \node[above] at (1,2.35) {$\ell=0$};

    % four corner nodes
    \foreach \P in {(0,0),(2,0),(2,2),(0,2)}
        \node[oldpt] at \P {};

    \node[right,align=left] at (2.4,1) {\scriptsize corner nodes only\\[-1mm]
    \scriptsize $n_{i,0}=4$};
\end{scope}

% -------------------------------------------------
% Right middle: ell = 1
% -------------------------------------------------
\begin{scope}[shift={(6.4,1.6)}]
    \draw[edgebox] (0,0) rectangle (2,2);
    \node[above] at (1,2.35) {$\ell=1$};

    % old points: corners
    \foreach \P in {(0,0),(2,0),(2,2),(0,2)}
        \node[oldpt] at \P {};

    % new points: edge midpoints
    \foreach \P in {(1,0),(2,1),(1,2),(0,1)}
        \node[newpt] at \P {};

    \node[right,align=left] at (2.4,1) {\scriptsize add edge midpoints\\[-1mm]
    \scriptsize $n_{i,1}=8$};
\end{scope}

% -------------------------------------------------
% Right bottom: ell = 2
% -------------------------------------------------
\begin{scope}[shift={(6.4,-1.5)}]
    \draw[edgebox] (0,0) rectangle (2,2);
    \node[above] at (1,2.35) {$\ell=2$};

    % old points: level-0 and level-1 points
    \foreach \P in {(0,0),(2,0),(2,2),(0,2),(1,0),(2,1),(1,2),(0,1)}
        \node[oldpt] at \P {};

    % new points: quarter points on each edge
    \foreach \P in {(0.5,0),(1.5,0),(2,0.5),(2,1.5),(1.5,2),(0.5,2),(0,1.5),(0,0.5)}
        \node[newpt] at \P {};

    \node[right,align=left] at (2.4,1) {\scriptsize add quarter points\\[-1mm]
    \scriptsize $n_{i,2}=16$};
\end{scope}

% -------------------------------------------------
% Legend
% -------------------------------------------------
\begin{scope}[shift={(0,-3.0)}]
    \node[oldpt] at (0,0) {};
    \node[right] at (0.25,0) {\scriptsize nodes inherited from lower levels};

    \node[newpt] at (4.7,0) {};
    \node[right] at (4.95,0) {\scriptsize nodes newly added at the displayed level};
\end{scope}

\end{tikzpicture}
\caption{Boundary hierarchical nodal points $\{x_{i,\ell}^k\}_{k=1}^{n_{i,\ell}}$ for the local edge space $E_{i,\ell}$ on $\partial\Omega_i$. 
The left panel shows a coarse subdomain $\Omega_i$ and its coarse-edge partition 
$\partial\Omega_i=\bigcup_{j=1}^4 \overline{\Gamma_i^j}$. 
The right panels display the boundary nodal points for levels $\ell=0,1,2$. 
Blue points are inherited from lower levels, while red points are newly added at the displayed level.}
\label{fig:grid-points-local-new}
\end{figure}
Next, we define the local multiscale space over each coarse subdomain $\Omega_i$ by 
\begin{align}\label{eq:local-multiscale}
\mathcal{L}^{-1}_{h,i}(E_{i,\ell})
:= \text{span} \left\{\mathcal{L}^{-1}_{h,i}(\psi_{i,\ell}^k): 1 \leq k \leq  n_{i,\ell}\right\},
\end{align}
where $\mathcal{L}^{-1}_{h,i} (\psi_{i,\ell}^k):=v\in V_{h,i}$ is the solution to the following local problem
\begin{equation}
\label{eq:Li}
  \left\{ \begin{aligned}
         a_i(v,w)&=0 &&\forall w\in V_{h,i}^0,\\
          v&=\psi_{i,\ell}^k &&\mbox{on }\partial\Omega_i.
  \end{aligned}\right.
\end{equation}
By the construction, the dimension of $\mathcal{L}^{-1}_{h,i} (E_{i,\ell})$ is at most $2^{\ell+2}$ if $\Gamma_i=\emptyset$. In practice, $n_{i,\ell}$ local problems are solved by the FEM to obtain the local multiscale space $\mathcal{L}^{-1}_{h,i}(E_{i,\ell})$, which can be carried out in parallel and thus has a scalable computational time. 

Finally, the edge multiscale ansatz space $V_{h,0}$ is defined by the weighted prolongation $\tilde{\mathcal{R}}_i^T$,
\begin{align}\label{eq:global-multiscale}
V_{h,0} := \text{span} \left\{\tilde{\mathcal{R}}_{i}^T\mathcal{L}^{-1}_{h,i}(\psi_{i,\ell}^k): \, 1 \leq i \leq N,\, 1 \leq k \leq  n_{i,\ell}\right\}.
\end{align}
The construction of $V_{h,0}$ is summarized in Algorithm \ref{algorithm:wavelet}.

\begin{algorithm}[!t]
\caption{The construction of the edge multiscale space $V_{h,0}$.}
\label{algorithm:wavelet}
\begin{algorithmic}[1]
\Require The level parameter $\ell\in \mathbb{N}$; subdomain $\Omega_i$ and its coarse edges $\Gamma_i^j$, $j=1,\ldots,m_i$; the subspace $E_{i,\ell}^j\subset L^2(\Gamma_i^j)$ up to level $\ell$ on each coarse edge $\Gamma_i^j$.
\Ensure The global edge multiscale space $V_{h,0}$.
\State Construct the local edge space $E_{i,\ell}$ by \eqref{eq:local-edge}.
\State Compute the local multiscale space $\mathcal{L}_{h,i}^{-1}(E_{i,\ell})$ by \eqref{eq:local-multiscale}.
\State Construct the coarse space $V_{h,0}$ by \eqref{eq:global-multiscale}.
\end{algorithmic}
\end{algorithm}

\begin{lemma}[Well-posedness of the local problems \eqref{eq:Li}]\label{lem:loc-coercive}
Under Assumption \ref{ass:resolution}, the sesquilinear form $a_i(\cdot,\cdot):V(\Omega_i)\times V(\Omega_i)\to\mathbb{C}$ is coercive, i.e., for all $w_i\in V(\Omega_i)$ with $w_i=0$ on $\partial\Omega_i$, there holds
\begin{align*}
\Re a_i(w_i,w_i)
\geq \frac{1}{2\Ce}\norm{w_i}_{V(\Omega_i)}^2.
\end{align*}
Here, $\Ce$ is defined in \eqref{eq:cest}.
\end{lemma}

\begin{proof}
For all $w_i\in V(\Omega_i)$ with $w_i=0$ on $\partial\Omega_i$, using the definition of the sesquilinear form $a_i(\cdot,\cdot)$ and Assumption \ref{ass:resolution}, we derive
\begin{align*}
\Re a_i(w_i,w_i)&=\normL{\nabla w_i}{\Omega_i}^2-k^2\normL{w_i}{\Omega_i}^2
\geq \left(1-k^2d_i^2\Cpoin{\Omega_i}\right)\normL{\nabla w_i}{\Omega_i}^2\\
&\geq \Ce^{-1}\normL{\nabla w_i}{\Omega_i}^2.
\end{align*}
Furthermore, the definition of $\|\cdot\|_{V(\Omega_i)}$ implies, 
\begin{align*}
\|w_i\|^2_{V(\Omega_i)}&=\normL{\nabla w_i}{\Omega_i}^2+k^2\normL{w_i}{\Omega_i}^2
\leq \left(1+k^2d_i^2\Cpoin{\Omega_i}\right)\normL{\nabla w_i}{\Omega_i}^2\\
&\leq 2\normL{\nabla w_i}{\Omega_i}^2.
\end{align*}
The combination of the previous two inequalities leads to the desired assertion.
\end{proof}
Next we give the coarse projection operator $\mathcal{P}_0$ of level $\ell$: $V_h\to V_{h,0}$. Since $V_{h,0}$ is generated by the local multiscale space $\mathcal{L}^{-1}_{h,i}(E_{i,\ell})$ by means of the weighted prolongation $\tilde{\mathcal{R}}_i^T$, we only need to define local projection operator $\mathcal{P}_{i,\ell}$ at level $\ell$: $L^2(\partial\Omega_i)\to \mathcal{L}^{-1}_{h,i}(E_{i,\ell})$. Then $\mathcal{P}_{i,\ell}: L^2(\partial\Omega_i)\to \mathcal{L}_{h,i}^{-1}(E_{i,\ell})$ can be defined by 
\begin{align}\label{eq:projectionEDGE}
\mathcal{P}_{i,\ell}v:=\mathcal{L}_{h,i}^{-1}(\mathcal{I}_{i,\ell}(v)).
\end{align}
Here, $\mathcal{I}_{i,\ell}:L^2(\partial\Omega_i)\to E_{i,\ell}$ denotes the $L^2$-projection. 

Note that $(\mathcal{P}_{i,\ell}v)|_{\partial\Omega_i}$ is the $L^2$-projection on $E_{i,\ell}$:
\begin{align*}
(\mathcal{P}_{i,\ell}v)|_{\partial\Omega_i}
=\mathcal{I}_{i,\ell}(v).
\end{align*}
Since any $v\in V_h$ can be represented by
$v=\sum_{i=1}^N \tilde{{R}}_i^T v|_{\Omega_i}$,
the coarse projection $\mathcal{P}_0$ of level $\ell$: $ V_h\to V_0$ can be defined in terms of the local projection \eqref{eq:projectionEDGE},
\begin{align}\label{eq:glo-proj}
\mathcal{P}_0(v):=\sum_{i=1}^N\tilde{{R}}_i^T \left(\mathcal{P}_{i,\ell}v|_{\partial\Omega_i}\right).
\end{align}
%The coarse prolongation $\mathcal{P}_0^T$ of level $\ell$: $ V_0\to V_h$ is the natural injection. 
 
\begin{comment}
$\mathcal{P}_0^T$ of level $\ell$: $ V'_0\to V'_h$ can be defined by means of the dual operator, 
\begin{align*}
(\mathcal{P}_0^T v_0)(v)=v_0(\mathcal{P}_0 v) \text{ for all }
v_0\in V_0\text{ and }v\in V_h.
\end{align*}
\end{comment}

\begin{comment}
\begin{remark}
Similar resolution assumption as Assumption \ref{ass:resolution} is commonly seen in the literature, e.g.,  \cite{MR2812565}. If we take the coarse scale mesh grid $H\ll k^{-1}$, then $\Ce\approx 1$.
\end{remark}
\end{comment}

\section{Convergence analysis}\label{sec:converge}
Now we prove that the error transfer operator $I-B_{\op{hybrid}}A$ is contractive under explicit conditions on $(k,h,\bar{d},\delta)$ in \eqref{eq:overlap-diameter} and \eqref{eq:fine-mesh-h}-\eqref{eq:wavelet-level-ell}. This is achieved as follows. First, we prove the approximation properties of local multiscale space $\mathcal{L}^{-1}_{h,i}(E_{i,\ell})$ in Lemma \ref{thm:proj} by the concept of very weak solution, its {\it a priori} estimate and conforming Galerkin approximation established in Section \ref{sec:very-weak}. 

Next, we establish the approximation property of the edge multiscale ansatz space $V_{h,0}$ to the target Helmholtz harmonic linear space $V^{\partial}$ and to the discrete solution of the adjoint problem \eqref{eqn:weakformDual} in Lemma \ref{thm:last-app-harmonic} and Lemma \ref{thm:last}, respectively. Finally, we prove that the error transfer operator $I-B_{\op{hybrid}}A$ is contractive using the classical Schatz argument which essentially relies on the G\aa rding's inequality \eqref{eq:garding} and the approximation property of $V_{h,0}$ to the solution of the adjoint problem \eqref{eqn:weakformDual}. Note that similar results on the continuous level have been established in \cite{fu2021wavelet}. The discrete approximation properties of $V_{h,0}$ and solvablity of the coarse solver $A_0$ represent a crucial novelty of this work.
\subsection{Approximation properties of local multiscale space $\mathcal{L}^{-1}_{h,i}(E_{i,\ell})$}\label{sec:very-weak}
We establish in this section a local \textit{a priori} estimate on a discrete level, to the aim of proving Lemma \ref{thm:proj}. Throughout this section, $\Omega_i$ is one coarse subdomain as defined in Section \ref{sec:hybrid-2} for all $i=1,2,\cdots,N$, and $g\in V_{h,i}|_{\partial\Omega_i}$.
\subsubsection{\textit{a priori} estimate on the PDE level}
Let $g\in V_{h,i}|_{\partial\Omega_i}$, and $v\in H^{1/2}(\Omega_i)$ be the solution to the following problem:
\begin{equation}\label{eq:pde-very}
\left\{
\begin{aligned}
\mathcal{L}_i v:=\Delta v+k^2 v&=0&&\mbox{in }\Omega_i\\
%\frac{\partial v}{\partial n}+\mathrm{i}kv&=0 &&\mbox{on }\Gamma_i\\
 v&=g &&\mbox{on }\partial\Omega_i.
\end{aligned}
\right.
\end{equation}
%Here, $\Gamma_i=\partial \Omega_i\cap\partial\Omega$ and $\gamma_i=\partial\Omega_i\backslash\partial\Omega$. 
We introduce the local test space
\begin{align}\label{eq:test-space}
X(\Omega_i):=\{z\in H^1_0(\Omega_i;\mathbb{C}):\mathcal{L}_i z\in L^2(\Omega_i;\mathbb{C})\},
\end{align}
which is endowed with the graph norm $\|\cdot\|_{X(\Omega_i)}$,
\begin{align*}
\forall z\in X(\Omega_i):\|z\|_{X(\Omega_i)}^2=\|\nabla  z\|_{L^2(\Omega_i)}^2+\|\mathcal{L}_i z\|_{L^2(\Omega_i)}^2.
\end{align*}
Then we propose the following weak formulation corresponding to Problem \eqref{eq:pde-very}:
seeking $v\in L^2(\Omega_i;\mathbb{C})$ such that
\begin{align}\label{eq:nonstd-variational}
\int_{\Omega_i}v\cdot\mathcal{L}_i  \bar{z}\dx=\int_{\partial\Omega_i}g\cdot\frac{\partial \bar{z}}{\partial n}\;\mathrm{d}s\quad\text{ for all }z\in X(\Omega_i).
\end{align}
Next, we prove the nonstandard variational formulation \eqref{eq:nonstd-variational} is well-posed. To this end, one has to first derive the $L^2(\gamma_i)$-estimate of the normal trace $\frac{\partial z}{\partial n}$ for any $z\in X(\Omega_i)$. 
\begin{lemma}\label{thm:pw-Regularity}
Let the Resolution Assumption \ref{ass:resolution} be valid. Let $w\in L^2(\Omega_i)$ and let $z(w)\in X(\Omega_i)$ satisfy
\begin{equation}\label{eq:pde-dual}
\left\{\begin{aligned}
\mathcal{L}_i z(w):=\Delta z(w)+k^2 z(w)&=w && \text{ in } \Omega_i\\
%\frac{\partial z(w)}{\partial n}-\mathrm{i}kz(w)&=0 &&\mbox{ on }\Gamma_i\\
%\frac{\partial z}{\partial n}&=0 &&\mbox{ on }\partial Q_1^{\epsilon}\cap \omega_i.\\
z(w)&=0 &&\text{ on }\partial\Omega_i.
\end{aligned}\right.
\end{equation}
Then there holds
\begin{align*}
\left\|\frac{\partial z(w)}{\partial n}\right\|_{L^2(\partial\Omega_i)}\leq
& \Cw {\bar{d}}^{1/2}\normL{w}{\Omega_i}.
%\|\kappa\frac{\partial z}{\partial n}\|_{L^2(\partial\omega_i\cap D_0)}&\lesssim \normL{f}{\omega_i}.
\end{align*}
Here, $\Cw$ is a generic positive constant independent of $(k,h,H)$, and its value may vary between equations.
\end{lemma}
\begin{proof}
This proof follows from \cite{MR4860976}.
For the sake of simplicity, we denote $z(w)$ by $z$ in the proof.
An application of the Poincar\'{e} inequality implies
\begin{align}\label{eq:111}
\normL{z}{\Omega_i}\leq \Cpoinn{\Omega_i}{1/2}d_i\normL{\nabla z}{\Omega_i}.
\end{align}
Testing \eqref{eq:pde-dual} with $z$ and applying integration by parts, we arrive at
\begin{align*}
    \|\nabla z\|_{L^2(\Omega_i)}^2&= k^2\|z\|_{L^2(\Omega_i)}^2-\int_{\Omega_i}w\cdot\bar{z}\dx.
\end{align*}
Then an application of the Cauchy-Schwarz inequality leads to
\begin{align*}
    \|\nabla z\|_{L^2(\Omega_i)}^2&\leq k^2\|z\|_{L^2(\Omega_i)}^2+\|w\|_{L^2(\Omega_i)}\|z\|_{L^2(\Omega_i)}.
\end{align*}
Applying the Poincar\'{e} inequality \eqref{eq:111}, we obtain
\begin{align*}
\normL{\nabla z}{\Omega_i}^2\leq \Cpoin{\Omega_i}(d_ik)^2\normL{\nabla z}{\Omega_i}^2+\Cpoinn{\Omega_i}{1/2}d_i\normL{w}{\Omega_i}
\normL{\nabla z}{\Omega_i}.
\end{align*}
Thanks to the Resolution Assumption \eqref{ass:resolution}, this yields
\begin{align}\label{eq:gradient}
\normL{\nabla z}{\Omega_i}\leq\Ce\Cpoinn{\Omega_i}{1/2}d_i \normL{w}{\Omega_i}.
\end{align}
Furthermore, by \eqref{eq:111}, we arrive at
\begin{align}\label{eq:l2}
\normL{ z}{\Omega_i}\leq \Ce\Cpoin{\Omega_i}d_i^2\normL{w}{\Omega_i}.
\end{align}
Note that taking the $L^2(\Omega_i)$-norm on the first equation in \eqref{eq:pde-dual}, together with an application of the triangle inequality, implies
\begin{align*}
\normL{\Delta z}{\Omega_i}\leq k^2\normL{z}{\Omega_i}+\normL{w}{\Omega_i}.
\end{align*}
This, together with \eqref{eq:l2} and the Resolution Assumption \eqref{ass:resolution}, yields
\[
\normL{\Delta z}{\Omega_i}\leq \Big(\Ce +1\Big)\normL{w}{\Omega_i}.
\]
Note that the interface $\gamma_i$ has sufficient smoothness, then the $H^2$ regularity of the Laplacian on a convex domain \cite[Section 3.2]{MR775683} implies the following {\it a priori} estimate
\begin{align}\label{eq:H2-test}
\|z\|_{H^2(\Omega_i)}\leq \normL{\Delta z}{\Omega_i}\leq \Big(\Ce +1\Big)\normL{w}{\Omega_i}.
\end{align}
This, together with \eqref{eq:gradient} and applying interpolation between $H^1(\Omega_i)$ and $H^2(\Omega_i)$ yields the $H^{3/2+\eta}(\Omega_i)$ regularity estimate for $\eta\in (0,1/2)$, 
\begin{align}\label{eq:h3/2}
\normHp{z}{\Omega_i}{3/2+\eta}\lesssim(\Ce+1)\Cpoinn{\Omega_i}{1/4-\eta/2}d_i^{1/2-\eta}\normL{w}{\Omega_i}.
\end{align}
Since differentiation is continuous from $H^{3/2+\eta}(\Omega_i)$ to $H^{1/2+\eta}(\Omega_i)$, by the trace theorem, we have
\begin{align*}
\|\frac{\partial z}{\partial n}\|_{L^2(\partial\Omega_i)}&\lesssim \|\frac{\partial z}{\partial n}\|_{H^{\eta}(\partial\Omega_i)}\lesssim \|\nabla z\|_{H^{1/2+\eta}(\Omega_i)}\lesssim \normHp{z}{\Omega_i}{3/2+\eta},
\end{align*}
which, together with \eqref{eq:h3/2}, the Resolution Assumption \ref{ass:resolution} and taking $\eta$ sufficiently close to 0, proves the desired assertion.
\end{proof}
Thanks to this \textit{a priori} estimate presented in Theorem \ref{thm:pw-Regularity}, we are ready to state the well-posedness of the nonstandard variational formulation \eqref{eq:nonstd-variational}. To begin, we introduce a sesquilinear form $c_i(\cdot,\cdot)$ on $L^2(\Omega_i;\mathbb{C})\times L^2(\Omega_i;\mathbb{C})$ and a linear form $b_i(\cdot)$ on $L^2(\Omega_i;\mathbb{C})$, defined by
\begin{equation}\label{eq:linear-form-local}
    \begin{aligned}       c_i(w_1,w_2)&:=\int_{\Omega_i}w_1\cdot\overline{w}_2\dx&&\text{for all }w_1,w_2\in L^2(\Omega_i)\\     b_i(w)&:=\int_{\partial\Omega_i}g\cdot\frac{\partial \overline{z(w)}}{\partial n}\;\mathrm{d}s&& \text{for all }w\in L^2(\Omega_i),
    \end{aligned}
    \end{equation}
with $z(w)$ being the unique solution to \eqref{eq:pde-dual}. Then the nonstandard variational form \eqref{eq:nonstd-variational} is equivalent to seeking $v\in L^2(\Omega_i)$ such that
\begin{equation}\label{eq:equivalence}
c_i(v,w)=b_i(w)\quad\text{for all }w\in L^2(\Omega_i).
\end{equation}
Note that testing \eqref{eq:pde-dual} with $\xi\in V(\Omega_i)$, we obtain for all $w\in L^2(\Omega_i;\mathbb{C})$, there holds
\begin{align}\label{eq:normtrace-cont}
\left\langle \xi|_{\partial\Omega_i},\frac{\partial z(w)}{\partial n}\right\rangle_{\partial\Omega_i}=
c_i(\xi,w)+a_i(\xi,z(w)), \qquad \forall \xi\in V(\Omega_i).
\end{align}
Here, $\langle \cdot,\cdot\rangle_{\partial\Omega_i}$ denotes the $L^2(\partial\Omega_i;\mathbb{C})$ inner product.
\begin{theorem}\label{thm a2}
Given $g\in V_{h,i}|_{\partial\Omega_i}$, let v be the solution to problem \eqref{eq:nonstd-variational} and let the test space $X(\Omega_i)$ be defined in \eqref{eq:test-space}. Then the nonstandard variational form \eqref{eq:nonstd-variational} is well-posed. Moreover, there holds,
\begin{equation*}
        \|v\|_{L^2(\Omega_i)}\leq \Cw d_i^{1/2}\|g\|_{L^2(\partial\Omega_i)}.
    \end{equation*}
\end{theorem}
\begin{proof}
    To prove the well-posedness of the nonstandard variational form \eqref{eq:nonstd-variational},  we only need to show the boundedness of the linear form $b(\cdot)$, which follows from Theorem \ref{thm:pw-Regularity}, 
\begin{align}\label{eq: bounded b}
|b(w)|\leq \Cw d_i^{1/2}\|g\|_{L^2(\partial\Omega_i)}\normL{w}{\Omega_i}.
\end{align}
Then the equivalence of problems \eqref{eq:equivalence} and \eqref{eq:nonstd-variational} implies the desired well-posedness of \eqref{eq:nonstd-variational}. 

Next, given $w\in L^2(\Omega_i)$, we obtain from \eqref{eq: bounded b} and \eqref{eq:equivalence}:
    \begin{align*}       \int_{\Omega_i}v\cdot\overline{w}\dx=c(v,w)&=b(w)\\
    &\leq \Cw d_i^{1/2}\|g\|_{L^2(\partial\Omega_i)}\normL{w}{\Omega_i}.
    \end{align*}
    Then this completes the proof by taking $w:=v$.
\end{proof}

\subsubsection{Local \textit{a priori} estimate on a discrete level}
In this section, we first propose the finite element approximation to Problem \eqref{eq:pde-very} by means of the conforming nonstandard variational formulation \eqref{eq:nonstd-variational} inspired by \cite{MR2084239}, which provides a finite element approximation to the very weak solution for Poisson problems. Then we analyze its $L^2$-stability.

Given $w\in L^2(\Omega_i;\mathbb{C})$. First, we seek $z_{h,i}(w)\in V_{h,i}^0$, which corresponds to conforming Galerkin approximation to \eqref{eq:pde-dual}, satisfying
\begin{align}\label{eq:very-weak-discrete}
a_i(\xi_{h,i},z_{h,i}(w))=-c_i(\xi_{h,i},w), \quad \forall \xi_{h,i}\in V^0_{h,i}.
\end{align}
Let Assumption \ref{ass:resolution} hold, then a combination of coerciveness of $a_i(\cdot,\cdot)$ as in Lemma \ref{lem:loc-coercive}, boundedness of $a_i(\cdot,\cdot)$ \eqref{eq:sesqui-bound-local}, the a priori estimate \eqref{eq:H2-test} and the stability of  $\Pi_{h,i}$ \eqref{eq:nodal-interp}, implies this conforming Galerkin approximation error, 
\begin{align}\label{eq:dual-local-err}
\norm{z(w)- z_{h,i}(w)}_{V(\Omega_i)}\leq 2\Ce(1+\Ce)\Const{I}\Const{bdd}(kd_i)^{-1}h\normL{w}{\Omega_i}.
\end{align}
Next, we introduce $\lambda_{h,i}(w)\in V_{h,i}|_{\partial\Omega_i}$ to approximate $\frac{\partial z(w)}{\partial n}$, s.t., 
\begin{align}\label{eq:discrete-flux}
\langle \xi_{h,i}|_{\gamma_i},\lambda _{h,i}(w)\rangle_{\partial\Omega_i}=
c_i(\xi_{h,i},w)+a_i(\xi_{h,i},z_{h,i}(w)), \quad \forall \xi_{h,i}\in V_{h,i}.
\end{align}
We then introduce a linear from $b_{h,i}: L^2(\Omega_i;\mathbb{C})\to \mathbb{C}$ to approximate $b_i(\cdot)$, defined by
\begin{align}\label{eq:linear-discrete}
b_{h,i}(w)=\int_{\partial\Omega_i}g\overline{\lambda_{h,i}(w)}\ds.
\end{align}
Next, we estimate the error between $b_i(\cdot)$ and its approximation $b_{h,i}(\cdot)$, which is based on the stability of discrete extension for $g$. 
\begin{lemma}[Discrete extension of $g$] \label{lem:discrete-extension}
Let Assumption \ref{ass:resolution} and \eqref{eq:fine-mesh-cond} hold. For $g\in V_{h,i}|_{\partial\Omega_i}$, let $\widehat{v}_{h,i}\in V_{h,i}$ satisfy
\begin{equation}\label{eq:pde-approx-stab}
\left\{
\begin{aligned}
a_i( \widehat{v}_{h,i},w_i)&=0\quad&&\forall w_i\in V_{h,i}^0\\
 \widehat{v}_{h,i}&=g&&\text{on }\partial\Omega_i.
\end{aligned}
\right.
\end{equation}
Then there holds 
\begin{align*}
\norm{\widehat{v}_{h,i}}_{V(\Omega_i)}&\leq \Const{weak}(kd_i)^{-1}h^{-1/2}\| g\|_{L^2( \partial\Omega_i)}.
\end{align*}
Here, %the positive constants $\Ce$ and $\Const{bdd}$ are defined in \eqref{eq:cest} and \eqref{eq:sesqui-bound}, and 
$\Const{weak}$ is a positive constant independent of $(k,h,\bar{d},\delta)$. 
\end{lemma}
\begin{proof}
Since $\widehat{v}_{h,i}\in V_{h,i}$ is a finite element function, we can represent it using the nodal basis functions. Let $\mathcal{N}(\Bar{\Omega}_i)$ be the set consisting of all finite element nodes on $\Bar{\Omega}_i$ with $\mathcal{N}^{\partial}_i\subset\mathcal{N}(\Bar{\Omega}_i)$ consisting of those located on ${\partial\Omega_i}$. 
    Since $\widehat{v}_{h,i}\in V_{h,i}$, we can decompose it into $\widehat{v}_{h,i}:=\widehat{v}_{h,i}^0+\widehat{v}_{h,i}^{\partial}$, where $\widehat{v}_{h,i}^0(x)=0$ for all $x\in\mathcal{N}^{\partial}_i$ and $\widehat{v}_{h,i}^{\partial}(x)=0$ for all $x\in \mathcal{N}(\Bar{\Omega}_i)\backslash\mathcal{N}^{\partial}_i$.

First, we derive an estimate of $\Vert \widehat{v}_{h,i}\Vert_{V(\Omega_i)}$ by $\Vert \widehat{v}_{h,i}^{\partial}\Vert_{V(\Omega_i)}$. By Lemma \ref{lem:loc-coercive}, the boundedness of $a_i(\cdot,\cdot)$ \eqref{eq:sesqui-bound-local} and equation \eqref{eq:pde-approx-stab}, we have
\begin{equation*}
\begin{aligned}
\Vert \widehat{v}_{h,i}-\widehat{v}_{h,i}^{\partial}\Vert_{V(\Omega_i)}^2
&\leq 2\Ce \Re a_i(\widehat{v}_{h,i}-\widehat{v}_{h,i}^{\partial},\widehat{v}_{h,i}-\widehat{v}_{h,i}^{\partial})\\
&=-2\Ce \Re a_i(\widehat{v}_{h,i}^{\partial},\widehat{v}_{h,i}-\widehat{v}_{h,i}^{\partial})\\
&\leq 2\Ce{\rm C_{bdd}}(1+(kd_i)^{-1})\Vert \widehat{v}_{h,i}^{\partial}\Vert_{V(\Omega_i)}\Vert \widehat{v}_{h,i}-\widehat{v}_{h,i}^{\partial}\Vert_{V(\Omega_i)},
\end{aligned}
\end{equation*}
which implies 
\begin{align*}
\Vert \widehat{v}_{h,i}-\widehat{v}_{h,i}^{\partial}\Vert_{V(\Omega_i)}
\leq 2\Ce{\rm C_{bdd}}(1+(kd_i)^{-1})\Vert \widehat{v}_{h,i}^{\partial}\Vert_{V(\Omega_i)}.
\end{align*}
Then an application of the triangle inequality leads to 
\begin{equation}\label{part 4}
    \Vert \widehat{v}_{h,i}\Vert_{V(\Omega_i)}\leq (1+2\Ce{\rm C_{bdd}}(1+(kd_i)^{-1}))\Vert \widehat{v}_{h,i}^{\partial}\Vert_{V(\Omega_i)}.
\end{equation}
Next, we estimate $ \Vert \widehat{v}_{h,i}^{\partial}\Vert_{V(\Omega_i)}$. Using the definition of $k$-weighted norm, inverse estimate of the finite element functions and \cite[equation (2.5)]{MR4890903}, and noting that $\widehat{v}_{h,i}^{\partial}$ vanished on all nodes except for those on $\gamma_i$ and $hk\ll 1$ due to \eqref{eq:fine-mesh-cond}, we can obtain
    \begin{equation*}%\label{part 1}
        \begin{aligned}
    \Vert \widehat{v}_{h,i}^{\partial}\Vert_{V(\Omega_i)}^2&=\Vert\nabla \widehat{v}_{h,i}^{\partial}\Vert^2_{L^2(\Omega_i)}+k^2\Vert  \widehat{v}_{h,i}^{\partial}\Vert^2_{L^2(\Omega_i)}\\
    &\lesssim(h^{-2}+k^2)\Vert  \widehat{v}_{h,i}^{\partial}\Vert^2_{L^2(\Omega_i)}\\
    &\lesssim h^{d-2}\sum_{x\in \mathcal{N}^{\partial}_i}(\widehat{v}_{h,i}^{\partial}(x))^2.
    \end{aligned}
    \end{equation*}
Together with Assumption  \ref{ass:resolution} and the equivalence of the discrete $L^2$ norm
\begin{align*}
\Vert v_h\Vert_{L^2_h(\partial\Omega_i)}^2:=h^{d-1}\sum_{x\in\mathcal{N}_i^{\partial}}|v_h(x)|^2\sim \Vert v_h\Vert_{L^2(\partial\Omega_i)}^2,
\end{align*}    
this leads to 
    \begin{equation}\label{part 1}
        \begin{aligned}
        \Vert 
        \widehat{v}_{h,i}^{\partial}\Vert_{V(\Omega_i)}
    &\lesssim h^{-1/2}\Vert  \widehat{v}_{h,i}^{\partial}\Vert_{L^2(\partial\Omega_i)}\\
    &= h^{-1/2}\Vert  g\Vert_{L^2(\partial\Omega_i)}.
    \end{aligned}
    \end{equation}
Insert this estimate into \eqref{part 4}, this proves the desired assertion. 
\end{proof}
\begin{lemma}\label{lem:err-linear-form}
Let Assumption \ref{ass:resolution} and \eqref{eq:fine-mesh-cond} hold. Given $g\in V_{h,i}|_{\partial\Omega_i}$. There holds,
\begin{align*}
|b_i(w)-b_{h,i}(w)|\leq\Const{weak}(kd_i)^{-3}h^{1/2}\normL{g}{\partial\Omega_i}\normL{w}{\Omega_i},\quad\forall w\in L^2(\Omega_i;\mathbb{C}).
\end{align*}
\end{lemma}
\begin{proof}
By definitions \eqref{eq:linear-form-local} and \eqref{eq:linear-discrete}, and testing \eqref{eq:normtrace-cont} and \eqref{eq:discrete-flux} with $\widehat{v}_{h,i}$ defined in Lemma \ref{lem:discrete-extension}, we derive
\begin{align*}
|b_i(w)-b_{h,i}(w)|&=\left|\int_{\partial\Omega_i}g\cdot\left(\frac{\partial \overline{z(w)}}{\partial n}-\overline{\lambda_{h,i}(w)}\right)\;\mathrm{d}s
\right|\\
&=\left| a_i(\widehat{v}_{h,i},z(w)-z_{h,i}(w))\right|.
\end{align*}
Then using the boundedness of $a_i(\cdot,\cdot)$ \eqref{eq:sesqui-bound-local}, the conforming Galerkin approximation estimate \eqref{eq:dual-local-err}, and Lemma \ref{lem:discrete-extension}, we obtain 
\begin{align*}
|b_i(w)-b_{h,i}(w)|&\leq \Const{bdd}(1+(kd_i)^{-1})\norm{z(w)-z_{h,i}(w)}_{V(\Omega_i)}\norm{\widehat{v}_{h,i}}_{V(\Omega_i)}\\
&\leq\Const{weak}(kd_i)^{-3}h^{1/2}\normL{w}{\Omega_i}\normL{g}{\partial\Omega_i}.
\end{align*}
This completes the proof.
\end{proof}
Finally, we define the finite element approximation to \eqref{eq:nonstd-variational}. Let
 $v_{h,i}\in V_{h,i}\subset L^2(\Omega_i;\mathbb{C})$ satisfy
 \begin{equation}\label{eq:nonstd-variational-discrete}
%\left\{
\begin{aligned}
c_i(v_{h,i},{w}_{h,i})&=b_{h,i}(w_{h,i})\quad&&\text{ for all }w_{h,i}\in V_{h,i}.
%\frac{\partial v_{h,i}}{\partial n}+\mathrm{i}kv_{h,i}&=0 &&\mbox{ on }\partial \Omega_i\cap\partial\Omega.
%v_{h,i}&=\noteLi{g},&&\text{on }\partial\Omega_i\backslash\partial\Omega.
\end{aligned}
%\right.
\end{equation}
To analyze error estimate, we introduce an impractical finite element approximation scheme to \eqref{eq:equivalence}, which reads as seeking 
$\widetilde{v}_{h,i}\in V_{h,i}\subset L^2(\Omega_i;\mathbb{C})$, s.t., 
\begin{equation}\label{eq:equivalence-imprac}
c_i(\widetilde{v}_{h,i},w_{h,i})=b_i(w_{h,i})\quad\text{for all }w_{h,i}\in V_{h,i}.
\end{equation}
\begin{lemma}\label{L2 error}
Let Assumption \ref{ass:resolution} and \eqref{eq:fine-mesh-cond} hold. Given $g\in V_{h,i}|_{\partial\Omega_i}$. Let $v_{h,i}\in V_{h,i}$ be the solution to \eqref{eq:nonstd-variational-discrete}. Then there holds 
\begin{align*}
\|v_{h,i}\|_{L^2(\Omega_i)}\leq 
\Const{weak}\max\{d_i^{1/2},(kd_i)^{-3}h^{1/2}\}\normL{g}{\partial\Omega_i}.
\end{align*}
\end{lemma}
\begin{proof}
An application of the triangle inequality implies, 
\begin{align*}
\normL{v-v_{h,i}}{\Omega_i}\leq \normL{v-\widetilde{v}_{h,i}}{\Omega_i}+\normL{\widetilde{v}_{h,i}-v_{h,i}}{\Omega_i}
\end{align*}
with $\widetilde{v}_{h,i}$ defined in \eqref{eq:equivalence-imprac}. Hence, we only need to derive an upper bound for the two terms on the right hand side. 

To estimate the first term, \eqref{eq:equivalence} and \eqref{eq:equivalence-imprac} implies
    \begin{align*}
        \|v-\widetilde{v}_{h,i}\|_{L^2(\Omega_i)}^2&=c_i(v-\widetilde{v}_{h,i},v-\widetilde{v}_{h,i})\\
        &=c_i(v-\widetilde{v}_{h,i},v-w_{h,i})\\
        &\leq\|v-\widetilde{v}_{h,i}\|_{L^2(\Omega_i)}\|v-w_{h,i}\|_{L^2(\Omega_i)}\text{ for all }w_{h,i}\in V_{h,i}. %V_{h,i}^0,
    \end{align*} 
Hence, we derive
 \begin{align*}
\|v-\widetilde{v}_{h,i}\|_{L^2(\Omega_i)}\leq 
\min_{w_{h,i}\in V_{h,i}}\|v-w_{h,i}\|_{L^2(\Omega_i)}.
\end{align*}
Next, a combination of \eqref{eq:equivalence-imprac} and \eqref{eq:nonstd-variational-discrete}, yields, 
\begin{align*}
c_i(\widetilde{v}_{h,i}-v_{h,i},w_{h,i})=b_i(w_{h,i})-b_{h,i}(w_{h,i})\;\forall w_{h,i}\in V_{h,i}.
\end{align*}
Then by taking $w_{h,i}:=\widetilde{v}_{h,i}-v_{h,i}$ and applying Lemma \ref{lem:err-linear-form}, we derive
\begin{align*}
\|v-v_{h,i}\|_{L^2(\Omega_i)}\leq 
\min_{w_{h,i}\in V_{h,i}}\|v-w_{h,i}\|_{L^2(\Omega_i)}+\Const{weak}(kd_i)^{-3}h^{1/2}\normL{g}{\partial\Omega_i}.
\end{align*}
This, together with Theorem \ref{thm a2}, leads to the desired result.
\end{proof}
%\subsubsection{Equivalence of discrete solutions}
Next, we present the equivalence of \eqref{eq:nonstd-variational-discrete} and the following conforming Galerkin finite element scheme based upon the standard variational form, which reads as seeking $v_{h,i}\in V_{h,i}$, s.t., 
\begin{equation}\label{eq:pde-approx-2}
\left\{
\begin{aligned}
a_i( v_{h,i},\phi_{h,i})&=0,\quad&&\forall \phi_{h,i}\in V_{h,i}^{0},\\
 v_{h,i}&=g,&&\text{on }\partial\Omega_i.
\end{aligned}
\right.
\end{equation}
\begin{lemma}\label{lem:one2one}
Let Assumption \ref{ass:resolution} hold. The mapping : $V_{h,i}\to V_{h,i}^0\times V_{h,i}|_{\partial\Omega_i}$ such that $w_{h,i}\mapsto (z_{h,i},\lambda_{h,i})$ given by \eqref{eq:very-weak-discrete} and \eqref{eq:discrete-flux}, is bijective. Moreover, there holds
\begin{align*}
c_i(\xi_{h,i},w_{h,i})=\langle \xi_{h,i}|_{\partial\Omega_i},\lambda _{h,i}\rangle_{\partial\Omega_i}-a_i(\xi_{h,i},z_{h,i}), \quad \forall \xi_{h,i}\in V_{h,i}.
\end{align*}
\end{lemma}
\begin{proof}
This proof is similar to \cite[Proof to Lemma 5.1]{MR2084239}, and we skip it for the sake of simplicity.
\end{proof}
\begin{lemma}\label{lem:equiv}
Let Assumption \ref{ass:resolution} hold. Let $g\in V_{h,i}|_{\partial\Omega_i}$. Then Problem \eqref{eq:nonstd-variational-discrete} and Problem \eqref{eq:pde-approx-2} are equivalent.
\end{lemma}
\begin{proof}
On the one hand, let $v_{h,i}\in V_{h,i}$ be the solution to \eqref{eq:pde-approx-2}. For any $w_{h,i}\in V_{h,i}$, find $(z_{h,i},\lambda_{h,i})\in V_{h,i}^0\times V_{h,i}|_{\partial\Omega_i}$ by solving \eqref{eq:very-weak-discrete} and \eqref{eq:discrete-flux}. Inserting $\xi_{h,i}:=v_{h,i}$ into \eqref{eq:discrete-flux}, we derive
\begin{align*}
\langle v_{h,i}|_{\partial\Omega_i},\lambda _{h,i}\rangle_{\partial\Omega_i}=
c_i(v_{h,i},w_{h,i})+a_i(v_{h,i},z_{h,i}).
\end{align*}
Since $z_{h,i}\in V_{h,i}^0$, and since $v_{h,i}$ satisfies \eqref{eq:pde-approx-2}, then the second term vanishes and this leads to
\begin{align*}
c_i(v_{h,i},w_{h,i})=\langle g,\lambda _{h,i}\rangle_{\partial\Omega_i}=b_{h,i}(w_{h,i}),\;\forall w_{h,i}\in V_{h,i}.
\end{align*}
This shows that $v_{h,i}$ is the solution to Problem \eqref{eq:nonstd-variational-discrete}. 

On the other hand, let $v_{h,i}\in V_{h,i}$ be the solution to \eqref{eq:nonstd-variational-discrete}. For any $\lambda_{h,i}\in V_{h,i}|_{\partial\Omega_i}$ and $z_{h,i}\in V_{h,i}^0$, let $w_{h,i}\in V_{h,i}$ be defined
\begin{align*}
c_i(\xi_{h,i},w_{h,i})=\langle \xi_{h,i}|_{\partial\Omega_i},\lambda _{h,i}\rangle_{\partial\Omega_i}-a_i(\xi_{h,i},z_{h,i}), \quad \forall \xi_{h,i}\in V_{h,i}.
\end{align*}
Let $\xi_{h,i}:=v_{h,i}$, this leads to
\begin{align*}
c_i(v_{h,i},w_{h,i})=\langle v_{h,i}|_{\partial\Omega_i},\lambda _{h,i}\rangle_{\partial\Omega_i}-a_i(v_{h,i},z_{h,i}).
\end{align*}
Since $v_{h,i}\in V_{h,i}$ is the solution to \eqref{eq:nonstd-variational-discrete}, then
\begin{align*}
c_i(v_{h,i},w_{h,i})=b_{h,i}(w_{h,i})=\langle g,\widehat{\lambda}_{h,i}\rangle_{\partial\Omega_i}.
\end{align*}
Here, $\widehat{\lambda}_{h,i}\in V_{h,i}$ is defined by \eqref{eq:discrete-flux} with $w:=w_{h,i}$. Lemma \ref{lem:one2one} implies that $\widehat{\lambda}_{h,i}=\lambda_{h,i}$, and therefore, a combination of these two identities leads to
\begin{align*}
a_i(v_{h,i},z_{h,i})=\langle v_{h,i}|_{\partial\Omega_i}-g,\lambda _{h,i}\rangle_{\partial\Omega_i},\;\forall z_{h,i}\in V_{h,i}^0,\;\lambda_{h,i}\in V_{h,i}|_{\partial\Omega_i}.
\end{align*}
Hence, $v_{h,i}$ is the solution to Problem \eqref{eq:pde-approx-2}. 
\end{proof}
Next, given $g\in V_{h,i}|_{\partial\Omega_i}$, we introduce $\widetilde{v}_{h,i}\in V_{h,i}$, s.t., 
\begin{equation}\label{eq:pde-approx-3}
\left\{
\begin{aligned}
a_i( \widetilde{v}_{h,i},\phi_{h,i})&=0,\quad&&\forall \phi_{h,i}\in V_{h,i}\text{ with }\phi_{h,i}|_{\gamma_i}=0,\\
 \widetilde{v}_{h,i}&=g|_{\gamma_i},&&\text{on }\gamma_i.
\end{aligned}
\right.
\end{equation}
Here, $\gamma_i\coloneqq \partial\Omega_i\backslash\partial\Omega$ denotes the inner boundary of the subdomain $\Omega_i$. 
Then the well-posedness of \eqref{eq:pde-approx-2} and \eqref{eq:pde-approx-3} implies that $\widetilde{v}_{h,i}$ satisfies \eqref{eq:pde-approx-2} with $g\coloneqq \widetilde{v}_{h,i}|_{\partial\Omega_i}$. Moreover, Lemma \ref{lem:equiv} indicates $\widetilde{v}_{h,i}$ satisfies \eqref{eq:nonstd-variational-discrete} with the same boundary data $g= \widetilde{v}_{h,i}|_{\partial\Omega_i}$.

\begin{theorem}\label{corollary:very-weak-discrete}
Let Assumption \ref{ass:resolution} and \eqref{eq:fine-mesh-cond} hold. Let $g\in V_{h,i}|_{\partial\Omega_i}$ and let $\widetilde{v}_{h,i}$ be the solution to \eqref{eq:pde-approx-3}. Then there exists a constant ${\rm C}_{{\rm weak}}$ independent of $(k,h,\bar{d},\delta)$ such that
\[
\normL{\chi_i\nabla \widetilde{v}_{h,i}}{\Omega_i}\leq {\rm C}_{{\rm weak}}\left(\delta_i^{-1}+k\right)\normL{ \widetilde{v}_{h,i}}{\Omega_i}.
\]
\end{theorem}
\begin{proof}
Recall that $\{\chi_i\}_{i=1}^N$ is a discrete partition of unity subordinate to the cover $\{\Omega_i\}_{i=1}^N$, which implies that $\chi_i=0$ on $\gamma_i$. Recall also that $\Pi_{h,i}$ denotes the nodal interpolation onto $V_{h,i}$.

%Since $g\in V_{h,i}|_{\partial\Omega_i}$, then Lemma \ref{lem:equiv} implies $v_{h,i}\in V_{h,i}$ is also the solution to \eqref{eq:pde-approx-2}. 

For the sake of simplicity, we abbreviate $\widetilde{v}_{h,i}$ with $v_{h,i}$ in this proof.  Testing \eqref{eq:pde-approx-3} with $\Pi_{h,i}(\chi_i^2v_{h,i})\in V_{h,i}$ and applying integration by parts, we arrive at
\begin{align}
\int_{\Omega_i}\chi_i^2|\nabla v_{h,i}|^2\dx&+2 \int_{\Omega_i}\nabla v_{h,i}\cdot \nabla\chi_i\chi_i \overline{v}_{h,i}\dx
-\int_{\Omega_i}k^2 \chi_i^2|v_{h,i}|^2 \dx\nonumber\\
&+ik\int_{\Gamma_i} v_{h,i}\cdot\chi_i^2 \overline{v}_{h,i}d{\rm s}+\int_{\Omega_i}\nabla v_{h,i}\cdot\nabla\left((\Pi_{h,i}-1)(\chi_i^2\overline{v}_{h,i})\right)\nonumber\\
&-\int_{\Omega_i}k^2 v_{h,i}(\Pi_{h,i}-1)(\chi_i^2 \overline{v}_{h,i})\dx+ik\int_{\Gamma_i} v_{h,i}\cdot(\Pi_{h,i}-1)(\chi_i^2 \overline{v}_{h,i})d{\rm s}=0.\label{eq:000000}
\end{align}
Since $v_{h,i}\in V_{h,i}$ and $\chi_i$ is real-valued, we obtain $\Im \int_{\Gamma_i} v_{h,i}\cdot\chi_i^2 \overline{v}_{h,i}d{\rm s}=0$. 
Hence, taking the real part of the identity \eqref{eq:000000} leads to
\begin{align}
\int_{\Omega_i}\chi_i^2|\nabla v_{h,i}|^2\dx&+2\Re \int_{\Omega_i}\nabla v_{h,i}\cdot \nabla\chi_i\chi_i \overline{v}_{h,i}\dx-\int_{\Omega_i}k^2 \chi_i^2|v_{h,i}|^2 \dx\nonumber\\
&+\Re\int_{\Omega_i}\nabla v_{h,i}\cdot\nabla\left((\Pi_{h,i}-1)(\chi_i^2\overline{v}_{h,i})\right)-\Re\int_{\Omega_i}k^2 v_{h,i}(\Pi_{h,i}-1)(\chi_i^2 \overline{v}_{h,i})\dx\nonumber\\
&-\Im k\int_{\Gamma_i} v_{h,i}\cdot(\Pi_{h,i}-1)(\chi_i^2 \overline{v}_{h,i})d{\rm s}=0.\label{eq:000}
\end{align}
Note that an application of the Cauchy-Schwarz inequality combining with \eqref{eq:pum-property} implies,
\begin{align}\label{eq:222-222}
\left|\int_{\Omega_i}\nabla v_{h,i}\cdot \nabla\chi_i\chi_i \overline{v}_{h,i}\dx\right|\leq \frac{1}{4}\normL{\chi_i\nabla v_{h,i}}{\Omega_i}^2+\Const{G}^2 \delta_i^{-2}\normL{ v_{h,i}}{\Omega_i}^2.
\end{align}
The stability of interpolation operator \eqref{eq:nodal-interp} implies 
\begin{align}\label{eq:lastli}
\normL{(\Pi_{h,i}-1)(\chi_i^2 v_{h,i})}{\Omega_i}\leq\Const{I} h\normL{\nabla(\chi_i^2 v_{h,i})}{\Omega_i},
\end{align}
together with the fact that $\nabla(\chi_i^2 v_{h,i})=2\chi_i v_{h,i}\nabla\chi_i+\chi_i^2\nabla v_{h,i}$, 
we obtain 
\begin{align}\label{eq:333}
\left| \int_{\Omega_i}k^2 v_{h,i}(\Pi_{h,i}-1)(\chi_i^2 \overline{v}_{h,i})\dx\right|
&\leq \frac{1}{4}\normL{\chi_i\nabla v_{h,i}}{\Omega_i}^2\nonumber\\
&+
\Const{I}k^2h\left(2\Const{G}\delta_i^{-1}+\Const{I}k^2h\right)
\normL{v_{h,i}}{\Omega_i}^2.
\end{align}
Moreover, for each $\tau\in\mathcal{T}_h\cap\Omega_i$, by standard interpolation error estimate \eqref{eq:nodal-interp}, there holds
\begin{align}\label{eq:d0000}
\normL{\nabla(\Pi_{h,i}-1)(\chi_i^2v_{h,i})}{\tau}\leq \Const{I}h\normL{D^2(\chi_i^2v_{h,i})}{\tau}.
\end{align}
Here, $D^2(\chi_i^2v_{h,i})=4\chi_i\op{sym}(\nabla\chi_i(\nabla v_{h,i})^T)+2\nabla\chi_i\nabla\chi_i^T v_{h,i}$. Hence, a combination of the H\"{o}lder's inequality, triangle inequality, discrete inverse estimate and \eqref{eq:d0000}, yields
\begin{align*}
&\quad\left|\int_{\Omega_i}\nabla v_{h,i}\cdot\nabla\left((\Pi_{h,i}-1)(\chi_i^2\overline{v}_{h,i})\right)\dx\right|\\
&=\left|\sum_{\tau\in\mathcal{T}_h\cap\Omega_i}\int_{\tau}\nabla v_{h,i}\cdot\nabla\left((\Pi_{h,i}-1)(\chi_i^2\overline{v}_{h,i})\right)\dx\right|\\
&\leq \sum_{\tau\in\mathcal{T}_h\cap\Omega_i}\normL{\nabla v_{h,i}}{\tau}\normL{\nabla(\Pi_{h,i}-1)(\chi_i^2{v}_{h,i})}{\tau}\\
&\leq \sum_{\tau\in\mathcal{T}_h\cap\Omega_i}\Const{inv}\Const{I}\normL{ v_{h,i}}{\tau}\normL{D^2(\chi_i^2v_{h,i})}{\tau}\\
&\leq \sum_{\tau\in\mathcal{T}_h\cap\Omega_i}\Const{inv}\Const{I}\normL{ v_{h,i}}{\tau}\normL{4\chi_i\op{sym}(\nabla\chi_i (\nabla v_{h,i})^T)+2\nabla\chi_i\nabla\chi_i^T v_{h,i}}{\tau}\\
%&\leq \sum_{\tau\in\mathcal{T}_h\cap\Omega_i}\Const{inv}\Const{I}\normL{ v_{h,i}}{\tau}\left(4\normL{\chi_i\nabla\chi_i (\nabla v_{h,i})^T}{\tau}+2\normL{\nabla\chi_i\nabla\chi_i^T v_{h,i}}{\tau}\right)\\
&\leq \sum_{\tau\in\mathcal{T}_h\cap\Omega_i}\Const{inv}\Const{I}\normL{ v_{h,i}}{\tau}\left(4\normL{\chi_i\nabla\chi_i (\nabla v_{h,i})^T}{\tau}+2\normL{\nabla\chi_i\nabla\chi_i^T v_{h,i}}{\tau}\right)\\
&\leq \sum_{\tau\in\mathcal{T}_h\cap\Omega_i}\Const{inv}\Const{I}\normL{ v_{h,i}}{\tau}\left(4\sqrt{d}\Const{G}\delta_i^{-1}\normL{\chi_i\nabla v_{h,i}}{\tau}+2d\Const{G}^2\delta_i^{-2}\normL{ v_{h,i}}{\tau}\right).
\end{align*}
Here, $\Const{inv}$ denotes a constant independent of $(k,h,\bar{d},\delta)$. 
Thus, we derive, 
\begin{align}
&\quad\left|\int_{\Omega_i}\nabla v_{h,i}\cdot\nabla\left((\Pi_{h,i}-1)(\chi_i^2\overline{v}_{h,i})\right)\dx\right|\nonumber\\
&\leq 4\sqrt{d}\sum_{\tau\in\mathcal{T}_h\cap\Omega_i}\Const{inv}\Const{I}\normL{ v_{h,i}}{\tau}\Const{G}\delta_i^{-1}\normL{\chi_i\nabla v_{h,i}}{\tau}+2{d}\sum_{\tau\in\mathcal{T}_h\cap\Omega_i}\Const{inv}\Const{I}\Const{G}^2\delta_i^{-2}\normL{ v_{h,i}}{\tau}^2\nonumber\\
&\leq 4\sqrt{d}\Const{G}\delta_i^{-1}\Const{inv}\Const{I}\normL{ v_{h,i}}{\Omega_i}\normL{\chi_i\nabla v_{h,i}}{\Omega_i}
+2{d}\Const{inv}\Const{I}\Const{G}^2\delta_i^{-2}\normL{ v_{h,i}}{\Omega_i}^2\nonumber\\
&\leq \frac{1}{4}\normL{\chi_i\nabla v_{h,i}}{\Omega_i}^2
+\left(16d\Const{G}^2\delta_i^{-2}
\Const{inv}^2\Const{I}^2
+2{d}\Const{inv}\Const{I}\Const{G}^2\delta_i^{-2}\right)\normL{ v_{h,i}}{\Omega_i}^2.\label{eq:111-111}
\end{align}
Next, we estimate the last term in \eqref{eq:000}. An application of the H\"{o}lder's inequality, the discrete inverse estimate and the trace inequality leads to
\begin{align*}
&\left|\int_{\Gamma_i} v_{h,i}\cdot(\Pi_{h,i}-1)(\chi_i^2 \overline{v}_{h,i})d{\rm s}\right|\\&\leq d_i^{-1/2}\normL{v_{h,i}}{\Omega_i}
\left(d_i^{-1/2}\normL{(\Pi_{h,i}-1)(\chi_i^2 {v}_{h,i})}{\Omega_i}+d_i^{1/2}\normL{\nabla (\Pi_{h,i}-1)(\chi_i^2 {v}_{h,i})}{\Omega_i}\right).
\end{align*}
Then using \eqref{eq:lastli}, \eqref{eq:333}, the Young's inequality and the basic fact that $d_i\geq h$ and $\delta_i\geq h$, we deduce
\begin{align}\label{eq:lastli2}
k\left|\int_{\Gamma_i} v_{h,i}\cdot(\Pi_{h,i}-1)(\chi_i^2 \overline{v}_{h,i})d{\rm s}\right|&\leq \frac{1}{8}\normL{\chi_i\nabla v_{h,i}}{\Omega_i}^2
+2\Const{I}\Const{G}\left(1+d\Const{G}\right)\delta_i^{-1}k\normL{v_{h,i}}{\Omega_i}^2\\&
+4\Const{I}^2\left(1+16d\Const{G}^2\right)k^2\normL{v_{h,i}}{\Omega_i}^2.\nonumber
\end{align}
Plug \eqref{eq:222-222}, \eqref{eq:333}, \eqref{eq:111-111} and \eqref{eq:lastli2} into \eqref{eq:000}, combining with \eqref{eq:fine-mesh-cond}, we obtain
\begin{align*}
\normL{\chi_i\nabla v_{h,i}}{\Omega_i}^2
\lesssim \left(\delta_i^{-2}+k^2\right)\|v_{h,i}\|_{L^2(\Omega_i)}^2
.
\end{align*}
\end{proof} 

%\begin{lemma}[Discrete error in local solver \eqref{eq:Li-local}]
%xx
%\end{lemma}
\begin{corollary}[Approximation properties of the local multiscale space $\mathcal{L}^{-1}_{h,i}(E_{i,\ell})$] \label{thm:proj}
Let Assumption \ref{ass:resolution} and \eqref{eq:fine-mesh-cond} hold. For all $v_i^{\partial}\in V^{\partial}_i$ given by \eqref{eq:local-space}, let $e_{h,i}\in V_{h,i}$ satisfy
\begin{equation}\label{eq:pde-approx}
\left\{
\begin{aligned}
a_i( e_{h,i},w_i)&=0,\quad&&\forall w_i\in V_{h,i}^0,\\
 e_{h,i}&=v_{i}^{\partial}-\mathcal{P}_{i,\ell}(v_{i}^{\partial})&&\text{on }\partial\Omega_i.
\end{aligned}
\right.
\end{equation}
Then there holds
\begin{comment}
    \begin{align*}
\norm{e_{h,i}}_{V(\Omega_i)}&\leq \Const{weak}(1+2\Ce{\rm C_{bdd}})\sqrt{\frac{H}{h}}2^{-\ell/2}\| v_{i}^{\partial}\|_{H^1( \Omega_i)}.
%\normL{\nabla(\chi_i{e}_{h,i})}{\Omega_i}
%&\leq\Cw\Ce 2^{-\ell/2}\delta^{-1} d_i\| v_{i}^{\partial}\|_{H^1( \Omega_i)}\noteLi{+2\Ce (1+(Hk)^{-1})\min_{v_{h,i}\in V_{h,i}}\norm{e_i-v_{h,i}}_{V(\Omega_i)}}.
\end{align*}
\end{comment}
\begin{align*}
\|\chi_ie_{h,i}\|_{V(\Omega_i)}\leq\Const{weak} \eta_i d_i(\delta_i^{-1}+k)2^{-\ell/2}\|v_{i}^{\partial}\|_{H^1(\Omega_i)}
\end{align*}
with the subdomain-wise quantity $\eta_i$ defined by
\begin{align*}
\eta_i\coloneqq \max\left\{1,d_i^{-1/2}(kd_i)^{-3}h^{1/2}\right\}.
\end{align*}
\end{corollary}
\begin{proof}
This assertion follows from Proposition \ref{prop:approx-wavelets}, Lemma \ref{L2 error} and Theorem \ref{corollary:very-weak-discrete}.
\end{proof}
\subsection{Contractive property of $I-B_{\op{hybrid}}A $}
Next, we derive approximation properties of the edge multiscale ansatz space $V_{h,0}$ to the target Helmholtz harmonic linear space $V^{\partial}$ and to the solution of the adjoint problem \eqref{eqn:weakformDual}. Similar to \cite[Theorem 4.2]{MR4860976}, the proof is based on the global-to-local splittings, but utilizing a weighted prolongation $\tilde{R}^T_i$ instead of the partition of the unity functions $\chi_i$. 
\begin{lemma}[Approximation properties of $V_{h,0}$ to $V^{\partial}$] \label{thm:last-app-harmonic}
Let Assumption \ref{ass:resolution} and \eqref{eq:fine-mesh-cond} hold.
For any $v^{\partial}\in V^{\partial}$ defined in \eqref{eq:glo-edge-space}, there exists $\Const{\rm ap}$  independent of $(k,h,\bar{d},\delta)$ such that
\begin{equation}\label{eq:waveletErrconv1}
\|v^{\partial}-\mathcal{P}_0 v^{\partial}\|_V
\leq \Const{\rm ap}\bar{\eta}\bar{d}(\delta^{-1}+k)2^{-\ell/2}\left(\sum_{i=1}^N\| v_{i}^{\partial}\|_{H^{1}( \Omega_i)}^2\right)^{1/2}.
\end{equation}
Here, we denote
\begin{align*}
\bar{\eta}\coloneqq \max_{i=1,\cdots,N}\left\{\eta_i\right\}.
\end{align*}
%\begin{equation}\label{eq:waveletErrconv1}
%\|v^{\partial}-\mathcal{P}_0 v^{\partial}\|_V
%\leq \Const{\rm ap} \Ce\delta^{-1} H2^{-\ell/2}\|{v^{\partial}}\|_{V}.
%\end{equation} 
\end{lemma}
	\begin{proof}
 By \eqref{eq:glo-edge-space} and  \eqref{eq:glo-proj}, each $v^{\partial}\in V^{\partial}$ and its coarse projection $\mathcal{P}_0(v^{\partial})$ at level $\ell$ can be expressed 
\begin{align}\label{eq0}
v^{\partial}&:=\sum_{i=1}^N \tilde{R}^T_{i}v_{i}^{\partial}\quad \mbox{and}\quad \mathcal{P}_0(v^{\partial})=\sum_{i=1}^N\tilde{{R}}_i^T \left(\mathcal{P}_{i,\ell}v^{\partial}|_{\partial\Omega_i}\right).
\end{align}
%then together with \eqref{eq0} and Theorem \ref{thm:pum-2}, we derive
%\begin{align}
%\|v^{\partial}-\mathcal{P}_0 v^{\partial}\|_{L^2(\Omega)}
%&\lesssim \left(\sum_{i=1}^N \|v^{\partial}_i-\mathcal{P}_{i,\ell}v^{\partial}\|_{L^2(\Omega_i)}^2
%+h^2\left\|\nabla(\chi_i(v^{\partial}_i-\mathcal{P}_{i,\ell}v^{\partial}|_{\partial\Omega_i}))\right\|_{L^2(\Omega_i)}^2\right)^{1/2},
%\label{eq3}\\
%\|\nabla \left(v^{\partial}-\mathcal{P}_0 v^{\partial}\right)\|_{L^2(\Omega)}
%&\leq \left(\sum_{i=1}^N \delta_i^{-2}\|v^{\partial}_i-\mathcal{P}_{i,\ell}v^{\partial}|_{\partial\Omega_i}\|_{L^2(\Omega_i)}^2
%+\left\|\chi_i\nabla \left(v^{\partial}_i-\mathcal{P}_{i,\ell}v^{\partial}|_{\partial\Omega_i}\right)\right\|_{L^2(\Omega_i)}^2\right)^{1/2}
%\label{eq4}.
%	\end{align}
\begin{comment}
    	\begin{align}
		\|v^{\partial}-\mathcal{P}_0 v^{\partial}\|_{L^2(\Omega)}&\leq \sqrt{\Cov}C_{\infty}\left(\sum_{i=1}^N\|v^{\partial}_i-\mathcal{P}_{i,\ell}v^{\partial}|_{\partial\Omega_i}\|_{L^2(\Omega_i)}^2\right)^{1/2},\label{eq3}\\
		\|\nabla \left(v^{\partial}-\mathcal{P}_0 v^{\partial}\right)\|_{L^2(\Omega)}&\leq \sqrt{2\Cov}\left(\sum_{i=1}^NC_{\text{G}}^2\delta_i^{-2}\|v^{\partial}_i-\mathcal{P}_{i,\ell}v^{\partial}|_{\partial\Omega_i}\|_{L^2(\Omega_i)}^2
		+\|\chi_i\nabla \left(v^{\partial}_i-\mathcal{P}_{i,\ell}v^{\partial}|_{\partial\Omega_i}\right)\|_{L^2(\Omega_i)}^2\right)^{1/2}\label{eq4}.
	\end{align}
\end{comment}
Let $e_h:=v^{\partial}-\mathcal{P}_0 v^{\partial}$. Then 
$e_h=\sum_{i=1}^N \tilde{\RR}^T_{i}e_{h,i}$, with $ e_{h,i}:=v^{\partial}_i-\mathcal{P}_{i,\ell} v^{\partial}|_{\partial\Omega_i}$.
Since $e_{h,i}$ satisfies \eqref{eq:pde-approx}, then a combination of Corollary \ref{thm:proj}, \eqref{eq0} and Lemma \ref{thm:pum-2}, leads to 
\begin{align*}
\|e_h\|_{V}&\leq\left(\Lambda(1+\Const{I}^2)\sum_{i=1}^N \|\chi_ie_{h,i}\|_{V(\Omega_i)}^2\right)^{1/2}\\
%\lesssim\left(C_{\infty}+k^{-1}\delta^{-1}C_{\op{G}}\right)\left(\sum_{i=1}^N \|e_{h,i}\|_{V(\Omega_i)}^2\right)^{1/2}\\
   &\lesssim \max_{i=1,\cdots,N}\{\eta_i d_i\}(\delta^{-1}+k)2^{-\ell/2}\left(\sum_{i=1}^N\| v_{i}^{\partial}\|_{H^{1}( \Omega_i)}^2\right)^{1/2}.
\end{align*}
%\begin{align*}
%\|e_h\|_{L^2(\Omega)}&\lesssim \left(\sum_{i=1}^N\|e_{h,i}\|_{L^2(\Omega_i)}^2+h^2\|\nabla (\chi_ie_{h,i})\|_{L^2(\Omega_i)}^2\right)^{1/2}
%\lesssim\Ce2^{-\ell/2}d_i \left(\sum_{i=1}^N\| v_{i}^{\partial}\|_{H^1( \Omega_i)}^2\right)^{1/2},\\%\label{eq5}\\
%\|\nabla e_h\|_{L^2(\Omega)}&\lesssim  \left(\sum_{i=1}^N\|\nabla (\chi_ie_{h,i})\|_{L^2(\Omega_i)}^2\right)^{1/2}
%\lesssim\Ce 2^{-\ell/2}\delta^{-1}{d_i}\left(\sum_{i=1}^N 
%\| v_{i}^{\partial}\|^2_{H^1( \Omega_i)}\right)^{1/2}.%\label{eq6}
%\end{align*}
%By combining this estimate with Assumption \ref{ass:resolution}, 
Thus, we prove \eqref{eq:waveletErrconv1}. %\Yueqi{The result I get is a little different from \eqref{eq:waveletErrconv1}. From the above two equations and since $\Vert v\Vert_V=\sqrt{\|\nabla v\|_{\Omega}^2+k^2\|v\|_{\Omega}^2}$, I get---- }\noteLi{use Assumption \ref{ass:resolution} you can get the claimed result.}
	%\begin{equation*}
		%\|v^{\partial}-\mathcal{P}_0 v^{\partial}\|_V\leq \Const{\ref{thm:last}} \Ce 2^{-\ell/2}\left(k^2H^2+H^2+(\noteLi{\frac{H}{h}})^2\right)^{1/2}\Vert v^{\partial}\Vert_V
	%\end{equation*}
\end{proof}
\begin{lemma}[Approximation properties of $V_{h,0}$ to the discrete solution of the adjoint problem \eqref{eqn:weakformDual}] \label{thm:last}
Let condition \eqref{eq:fine-mesh-cond} and Assumption \ref{ass:resolution} hold. 
For all $w_h\in V_h$, let $z_h\in V_h$ satisfy
\begin{align*}%\label{eqn:weakformDual}
a^*(z_h, v_h)=(w_h,v_h)_{\Omega}, \quad\forall v_h\in V_h,
\end{align*}
which corresponds to the conforming Galerkin approximation to the adjoint problem \eqref{eqn:weakformDual}. Assume further that
\begin{align}\label{eq:overlap-diameter}
\frac{\bar{d}}{\delta}\lesssim 1.
\end{align}
Then there exists $\Const{st}$  independent of $(k,h,\bar{d},\delta)$ such that
\begin{equation}\label{eq:waveletErrconvDual}
\inf\limits_{v_0\in V_{h,0}}\|z_h-v_0\|_V\leq
\Const{st}\bar{\eta}{\bar{d}}(\delta^{-1}+k)\left(\bar{d}+2^{-\ell/2}\right)
\|{w_h}\|_{L^2(\Omega)}.
\end{equation}
%\begin{equation}\label{eq:waveletErrconvDual}
%\inf\limits_{v_0\in V_{h,0}}\|z_h-v_0\|_V\leq
%\Const{\ref{thm:last}}\delta^{-1}{H}
%(H+\Ce^2 2^{-\ell/2})\|{w_h}\|_{L^2(\Omega)}.
%\end{equation}
\end{lemma}
\begin{proof}
%with $z_h^{i,\roma}:=z_h|_{\Omega_i}-(z_h^{i,\romb}+\overline{v^{i}}\dashint_{\Omega_i}w_h{\rm d}x)$ and $\dashint_{\Omega_i}v{\rm d}x:=|\Omega_i|^{-1}\int_{\Omega_i}v{\rm{d}}x$ denotes the average of the function $v\in L^1(\Omega_i)$ over each subdomain $\Omega_i$. 
Let $z_h^{i,\romb}\in V_{h,i}$ be the solution to the following problem 
	\begin{equation*}
		\left\{
		\begin{aligned}
			a_i(z_h^{i,\romb},v_i)&=0,&&\forall v_i\in V_{h,i}^0,\\
			z_h^{i,\romb}&=z_{h}|_{\partial\Omega_i},&&\mbox{on }\partial\Omega_i, 
		\end{aligned}
		\right.
	\end{equation*}
and let $z_h^{i,\roma}:=z_h|_{\Omega_i}-z_h^{i,\romb}$. 
Then a simple calculation reveals that $z_h^{i,\roma}\in V_{h,i}^0$ satisfies
\begin{equation}\label{eq:loc-bubble-id}	
a_i(z_h^{i,\roma},z_h^{i,\roma})=\left(z_h^{i,\roma},w_h\right)_{\Omega_i}.
\end{equation}

Together with \eqref{eq:pou}, we derive the global splitting 
\begin{align*}
z_h&=\sum_{i=1}^N\tilde{{R}}_i^{T}(z_h|_{\Omega_i})
=\sum_{i=1}^N\tilde{{R}}_i^{T}\left(z_h^{i,\roma}+z_h^{i,\romb}\right).
\end{align*}
	%\begin{align*}
	%	\Vert z\Vert_V&\lesssim(1+k)\sup_{v\in V}\frac{|\bar{a}(z,v)|}{\Vert v\Vert_V}=(1+k)\sup_{v\in V}\frac{|a(v,z)|}{\Vert v\Vert_V}\\
	%	&\lesssim(1+k)\sup_{v\in V}\frac{|(v,w)_{\Omega}|}{\Vert v\Vert_V}\lesssim(1+k)\sup_{v\in V}\frac{\Vert v\Vert_{L^2(\Omega)}\Vert w\Vert_{L^2(\Omega)}}{\Vert v\Vert_V}\\
	%	&\lesssim\frac{1+k}{k}\Vert w\Vert_{L^2(\Omega)}\lesssim\Vert w\Vert_{L^2(\Omega)}.
	%\end{align*} 
Let $v_0:=\sum_{i=1}^N\tilde{{R}}_i^T (\mathcal{P}_{i,\ell}z_h|_{\partial\Omega_i})\in V_{h,0}$ be an approximation to $z_h$. 
Direct computation gives 
	\begin{align}
		\Vert z_h-v_0\Vert^2_V&=\left\Vert \sum_{i=1}^N\tilde{{R}}_i^T z_h^{i,\roma}+\sum_{i=1}^N\tilde{{R}}_i^T(z_h^{i,\romb}-\mathcal{P}_{i,\ell}z_h^{i,\romb}|_{\partial\Omega_i})\right\Vert_V^2\nonumber\\
&\leq 2\left\Vert\sum_{i=1}^N\tilde{{R}}_i^T z_h^{i,\roma}\right\Vert_{V}^2
+2\left\Vert\sum_{i=1}^N\tilde{{R}}_i^T(z_h^{i,\romb}-\mathcal{P}_{i,\ell}z_h^{i,\romb}|_{\partial\Omega_i})\right\Vert_{V}^2.\label{eq2_4}
\end{align}
Next we derive upper bounds for each term. Lemma \ref{lem:loc-coercive}, together with the Poincar\'{e} inequality and the identity \eqref{eq:loc-bubble-id}, leads to
\begin{align*}
\frac{1}{2\Ce} \|z_h^{i,\roma}\|_{V(\Omega_i)}^2&\leq \normL{w_h}{\Omega_i}\Cpoin{\Omega_i}^{1/2}d_i\normL{\nabla z_h^{i,\roma}}{\Omega_i}\\
&\leq \normL{w_h}{\Omega_i}\Cpoin{\Omega_i}^{1/2}d_i\|z_h^{i,\roma}\|_{V(\Omega_i)}.
\end{align*}
Consequently, we obtain
\begin{align}\label{eq:local-bubble-new}
\|z_h^{i,\roma}\|_{V(\Omega_i)}\leq 2\Ce\Cpoin{\Omega_i}^{1/2}d_i\normL{w_h}{\Omega_i}. 
\end{align}

%By \cite[Lemma 4.1]{MR4860976}, we have 
% \begin{align}\label{eq:loc-bubble}
%d_i\Vert \nabla z_h^{i,\roma}\Vert_{L^2(\Omega_i)}+\Cpoin{\Omega_i}^{-1/2}\Vert z_h^{i,\roma}\Vert_{L^2(\Omega_i)}
%\leq \Cpoin{\Omega_i}^{1/2}d_i^2\Vert w_h\Vert_{L^2(\Omega_i)}.
% \end{align}
%This, together with \eqref{eq:glo-err-pum2}, yields
This, combining with the fact that $k\bar{d}\lesssim 1$,  $\frac{\bar{d}}{\delta}\lesssim 1$, Lemma \ref{thm:pum-2} and \cite[Equation (6.9)]{MR4890903}
\begin{align*}%\label{eq:local-estimate} 
\left\Vert\sum_{i=1}^N\tilde{{R}}_i^T z_h^{i,\roma}\right\Vert_{V}^2\lesssim 
\sum_{i=1}^N\Vert\nabla z_h^{i,\roma}\Vert^2_{L^2(\Omega_i)},
\end{align*}
yields
\begin{align}\label{eq:local-estimate} 
\left\Vert\sum_{i=1}^N\tilde{{R}}_i^T z_h^{i,\roma}\right\Vert_{V}\lesssim \bar{d}
\Vert w_h\Vert_{L^2(\Omega)}.
\end{align}
%\begin{align}\label{eq:local-estimate} 
%\left\Vert\sum_{i=1}^N\tilde{{R}}_i^T z_h^{i,\roma}\right\Vert_{V}\lesssim 
%\delta^{-1} d_i^2 \Vert w_h\Vert_{L^2(\Omega)}.
%\end{align}
Meanwhile, Corollary \ref{thm:proj} and the definition of $z_h^{i,\romb}$ indicate 
\begin{equation}\label{eq:lastsss}
\begin{aligned}
\Vert \chi_i(z_h^{i,\romb}-\mathcal{P}_{i,\ell}z_h^{i,\romb}|_{\partial\Omega_i})\Vert_{V(\Omega_i)}
\lesssim  \eta_i d_{i}(\delta_i^{-1}+k)2^{-\ell/2}\| z_h^{i,\romb}\|_{H^1( \Omega_i)}.
\end{aligned}
\end{equation}
%\begin{equation}%\label{eq:lastsss}
%\begin{aligned}
%\normL{z_h^{i,\romb}-\mathcal{P}_{i,\ell}z_h^{i,\romb}|_{\partial\Omega_i}}{\Omega_i}
%&\leq \Cw \Ce2^{-\ell/2}d_i\| z_h^{i,\romb}\|_{H^1( \Omega_i)},\\
%\normL{\nabla(\chi_i(z_h^{i,\romb}-\mathcal{P}_{i,\ell}z_h^{i,\romb}|_{\partial\Omega_i}))}{\Omega_i}
%&\leq \Cw \Ce2^{-\ell/2}\delta^{-1}d_i\| z_h^{i,\romb}\|_{H^1( \Omega_i)}.
%\end{aligned}
%\end{equation}
Notice that combining \eqref{eq:local-bubble-new} with the identity $z_h|_{\Omega_i}=z_h^{i,\roma}+z_h^{i,\romb}$, and the triangle inequality, we obtain 
\begin{align*}%\label{eq:local-estimate2}
\left\Vert z_h^{i,\romb}\right\Vert_{H^1(\Omega_i)}
\leq \left\Vert z_h\right\Vert_{H^1(\Omega_i)}+\left\Vert z_h^{i,\roma}\right\Vert_{H^1(\Omega_i)}
\leq 
\left\Vert  z_h\right\Vert_{H^1(\Omega_i)}+
2\Ce \Cpoin{\Omega_i}^{1/2}d_i\Vert w_h\Vert_{L^2(\Omega_i)}.
\end{align*}
To derive the {\it a priori} estimate for $z_h$, Theorem \ref{thm_1}, Remark \ref{final:1111} and the condition \eqref{eq:fine-mesh-cond}, imply
	\begin{align*}
		\Vert z_h\Vert_V
		\leq\Const{stab}\Vert w_h\Vert_{L^2(\Omega)}.\nonumber%\label{eq2_3}
\end{align*}
Combining these two inequalities above, we derive {a priori} estimate for the local Harmonic part,
\begin{align}\label{eq:local-estimate2}
\left\Vert z_h^{i,\romb}\right\Vert_{H^1(\Omega_i)}
\leq (\Const{stab}+2\Ce \Cpoin{\Omega_i}^{1/2}d_i)\Vert w_h\Vert_{L^2(\Omega_i)}.
\end{align}
Hence, a combination of Lemma \ref{thm:pum-2}, \eqref{eq:lastsss} and \eqref{eq:local-estimate2} leads to
\begin{align*}%
\left\Vert\sum_{i=1}^N\tilde{{R}}_i^T
\left(z_h^{i,\romb}-\mathcal{P}_{i,\ell}z_h^{i,\romb}|_{\partial\Omega_i}\right)\right
\Vert_{V}^2
&\lesssim \max_{i=1,\cdots,N}\{\eta_i^2\}{\bar{d}}^2(\delta^{-1}+k)^22^{-\ell}\sum_{i=1}^N\| z_h^{i,\romb}\|_{H^1(\Omega_i)}^2\nonumber\\
&\lesssim \max_{i=1,\cdots,N}\{\eta_i^2\}\bar{d}^2(\delta^{-1}+k)^22^{-\ell}\Vert w_h\Vert_{L^2(\Omega)}^2.
%&\lesssim\left(C_{\infty}^2+k^{-2}\delta^{-2}C_{\op{G}}^2\right)(1+2\Ce{\rm C_{bdd}})^2\frac{H}{h}2^{-\ell}H^2\Vert w_{h}\Vert_{L^2(\Omega)}^2,%\label{eq2_10}
\end{align*}
%\begin{align*}
%&\left\Vert\sum_{i=1}^N\tilde{{R}}_i^T
%(z_h^{i,\romb}-\mathcal{P}_{i,\ell}z_h^{i,\romb})\right
%\Vert_{V}^2
%\lesssim\Ce^2 2^{-\ell}\left(k^2H^2+ (\delta^{-1}H)^2\right)\sum_{i=1}^N\| z_h^{i,\romb}\|_{H^1(\Omega_i)}^2\nonumber\\
%\lesssim&\Ce^2 2^{-\ell}\left(k^2H^2+(\delta^{-1}H)^2\right)\left(\| z_h\|_{H^1(\Omega)}^2+\Ce^2 d_i^2\Vert w_h\Vert_{L^2(\Omega)}^2\right)
%\lesssim\Ce^4 2^{-\ell}(\delta^{-1}H)^2\Vert w_{h}\Vert_{L^2(\Omega)}^2,%\label{eq2_10}
%\end{align*}
Finally, we obtain the desired assertion by combining with \eqref{eq2_4} and \eqref{eq:local-estimate}.
\end{proof}
\begin{comment}
\begin{proposition}\label{prop:wavelet-basedconv}
Assume that the coarse mesh size $H$ and the level parameter $\ell\in \mathbb{N}_{+}$ satisfy
\begin{align}\label{eq:wavelet-level}
H\leq k^{-1}\frac{1}{4\Const{bdd}\Const{\ref{thm:last}}\Ce}\quad \text{ and }\quad\ell\geq 2\log_2(\Const{bdd}\Const{\ref{thm:last}}\Ce^2k\Co)+4.
\end{align}
Then for any $v_h\in V_h$, there holds
\begin{equation}\label{eq:waveletErrconv}
\begin{aligned}
\|\mathcal{A}^{-1}v_h-\mathcal{C}_0 v_h\|_V
%&:=\min\limits_{w\in  V_{{\rm off}}^{{\rm EW}}}\seminormE{u-w}{D}\\
&\leq 4 \Const{bdd} \Const{\ref{thm:last}} \Ce H\|v_h\|_{L^2(\Omega)}.
\end{aligned}
\end{equation}
\end{proposition}
\end{comment}

\begin{remark}[$\ell\leq \log_2\frac{\bar{d}}{h}$]\label{rem:hH}
Note that the fine-scale local trace space $V_h(\Gamma_{i}^j)$ is a refinement of $V_{\ell}(\Gamma_{i}^j)$ for $i=1,\cdots,N$ and $j=1,\cdots,m_i$, then $V_{\ell}(\Gamma_{i}^j)=V_h(\Gamma_{i}^j)$ in the case that $\ell\leq \log\frac{\bar{d}}{h}$. Consequently, $\mathcal{P}_{i,\ell}v_{i}^{\partial}=v_{i}^{\partial}$ for all $v_{i}^{\partial}\in V_{h,i}$. Then we obtain the trival estimate of \eqref{eq:waveletErrconv1} and \eqref{eq:waveletErrconvDual} by replacing the term $2^{-\ell/2}$ with zero. 
\end{remark}

\begin{theorem}[Contractivity of the operator $I-B_{\op{hybrid}}A$]\label{prop:wavelet-basedconv}
Let %{\color{blue} \st the fine grid mesh size} $h$ and
the diameter $d_i$ of each subdomain $\Omega_i$ satisfy %\eqref{eq:fine-mesh-cond} \ednote{Shihua: to delete} , 
Assumption \ref{ass:resolution} and \eqref{eq:overlap-diameter}. Moreover, assume there exists a small constant $c_0>0$ suffciently small, such that the fine grid mesh size $h$, the maximum diameter of the subdomains $\bar{d}$ and the level parameter $\ell\in \mathbb{N}_{+}$ satisfy the following conditions %\ednote{Shihua: we need to give specific suggestions how to set $\alpha, \beta$ }
\begin{align}
 h &\lesssim k^{-3/2},\label{eq:fine-mesh-h} \\
\bar{d}&=c_0 k^{-1},\label{eq:wavelet-level-H}\\
\ell&= \min\left\{2\lceil\log_2 \bar{d}\rceil,\lfloor\log_2\frac{\bar{d}}{h}\rfloor\right\}.\label{eq:wavelet-level-ell}
\end{align}
Then, given any $0< \rho <1$, there exist $k_0>0$ and $C>0$ sufficiently large such that, for all $k>k_0$ and $h\ge C k^{-3}$, it holds
\begin{equation}\label{eq:waveletErrconv}
\|I-B_{\op{hybrid}}A\|\leq \rho < 1.
\end{equation}
\end{theorem}
\begin{proof}
{Let $v_h\in V_h$ and let $v_h^{\partial}$ be defined in \eqref{eq:proj-edge-err}, then we derive $e_h:=(I-B_{\op{hybrid}}A) v_h = (I-R_0^TQ_0)v_h^{\partial}\in V_h$.} The
G\aa rding\rq{}s inequality in Theorem \ref{them:sesquilinear} implies
\begin{align*}
\|e_h\|_V^2&\leq \Re a(e_h,e_h)
+2k^2\normL{e_h}{\Omega}^2.
\end{align*}
{The representation above, together with the inclusion $V_{h,0}\subset V_h$ and Galerkin orthogonality, implies that $a(e_h, v_0)=0$ for all $v_0\in V_{h,0}$. Hence, we obtain}
\begin{align}\label{eq:222}
\|e_h\|_V^2
&\le \Re a(e_h,v_h^{\partial}-v_{0})
+2k^2\normL{e_h}{\Omega}^2, \quad\forall v_0\in V_{h,0}.
\end{align}
Next, we estimate $\normL{e_h}{\Omega}$ by the Aubin-Nitsche technique.%\ednote{Shihua: we may need to consider the modified duality argument} 
Let $z_h\in V_h$ be the solution to
\begin{align*}
a^*(z_h,v_h)=(e_h,v_h)_{\Omega}, \quad\forall v_h\in V_h. 
\end{align*}
For all $z_0\in V_{h,0}$, we obtain
\begin{align*}
\normL{e_h}{\Omega}^2
&=a^*(z_h,e_h)=\overline{a(e_h,z_h)}
=\overline{a(e_h,z_h-z_0)}.%\\
%&\leq \Const{bdd}\|u-u_{\text{ms},\ell}^{\text{EW}}\|_V\inf\limits_{z_{\text{ms},\ell}^{\text{EW}}\in V_{\text{ms},\ell}^{\text{EW}}}\|z-z_{\text{ms},\ell}^{\text{EW}}\|_V
\end{align*}
Furthermore, the estimates \eqref{eq:sesqui-bound} and \eqref{eq:waveletErrconvDual} lead to
\begin{align*}
\normL{e_h}{\Omega}^2
&\leq \Const{bdd}\|e_h\|_V\inf\limits_{z_0\in V_0}\|z_h-z_0\|_V\\
&\leq  \Const{bdd}\Const{st}\bar{\eta}\bar{d}(\delta^{-1}+k)\left(\bar{d}+2^{-\ell/2}\right) \|e_h\|_V \normL{e_h}{\Omega}.
\end{align*}
This implies 
\begin{align*}
\normL{e_h}{\Omega}
\leq  \Const{bdd}\Const{st}\bar{\eta}\bar{d}(\delta^{-1}+k)\left(\bar{d}+2^{-\ell/2}\right)\|e_h\|_V.
\end{align*}
Moreover, conditions \eqref{eq:overlap-diameter}, \eqref{eq:wavelet-level-H},  \eqref{eq:wavelet-level-ell} and Remark \ref{rem:hH} imply {$\bar{\eta}\bar{d}(\delta^{-1} +k) \lesssim 1$, and $\bar{d} +2^{-\ell/2} \le c_0k^{-1}$. Since $c_0>0$ is sufficiently small, then we obtain
\begin{align}\label{eq:schatz-condition}
\normL{e_h}{\Omega}\leq \tfrac{1}{2k}\|e_h\|_V,\quad \forall k\ge k_0.
\end{align}}
This and \eqref{eq:222} lead to
\begin{align*}
\|e_h\|_V^2\leq \Re a(e_h,v_h^{\partial}-v_0)+\tfrac{1}{2}\|e_h\|_V^2.
\end{align*}
Consequently, we obtain
$\|e_h\|_V^2\leq 2 \Re a(e_h,v_h^{\partial}-v_0)$.
Finally,  the boundedness of the sesquilinear form $a(\cdot,\cdot)$ in Theorem \ref{them:sesquilinear} implies
\begin{align*}
\|e_h\|_V\leq 2\Const{bdd} \inf_{v_0\in V_{h,0}}\|v_h^{\partial}-v_0\|_V.
\end{align*}
Then the approximation property \eqref{eq:waveletErrconv1} and the estimate \eqref{eq:vhpar} indicate
\begin{align*}
\|e_h\|_V\leq 2\Const{bdd}\Const{\rm ap}C_{\partial}\sigma \bar{\eta}\bar{d}({\delta}^{-1}+k)2^{-\ell/2} \|{u_h}\|_{V}.
\end{align*}
Combining with conditions \eqref{eq:overlap-diameter}, \eqref{eq:fine-mesh-h},
 \eqref{eq:wavelet-level-H} and \eqref{eq:wavelet-level-ell}, we obtain $\sigma\lesssim c_0^{-3/2} (kh)^{-1/2}$,  $\bar{\eta}\bar{d}(\delta^{-1} +k) \lesssim 1$, and $2^{-\ell/2} \le c_0k^{-1}$. Thus, we  arrive at
\begin{equation*}%\label{eq:waveletErrconv}
\|e_h\|_V
\lesssim \left(\tfrac{k^{ -3}}{h}\right)^{1/2}\|{v_h}\|_{V}.
\end{equation*}
This proves the desired assertion for $h \geq C k^{ -3}$ for a constant $C$ large enough.
\end{proof}
\begin{remark}[Well-posedness of coarse solver $A_0$]
Note that \eqref{eq:schatz-condition} guarantees well-posedness of the coarse problem $A_0$ \cite{MR373326}. 
\end{remark}

\section{Numerical tests}\label{sec:num}
%\textcolor{red}{$H\leq k^{-1-\alpha}$, $\alpha=0.1,0.2$, $\delta=H$}
We first verify two estimates used in Section \ref{sec:converge}. For each listed wavenumber $k$, we employ two settings of $(h,H)$. {\em Setting A} corresponds to 
$h \simeq k^{-3/2}$ and $ H=1/k$,  
where the notation $\simeq$ means that $h$ is chosen as close to $k^{-3/2}$ as possible.
{\em Setting B} takes a fixed number $m$ of grid-points per wavelength,
$h =2\pi/(mk)$ and $H=2/k$,
and we take $m=6\pi$, so that $h=1/(3k)$. The overlap is $\delta=h$ in Tests 1 and 2. Tests 3 and 4 examine convergence as the wavenumber and problem size increase at a fixed number of points per wavelength. Test 5 considers two- and three-dimensional media with discontinuous wave speed.

\subsection{Test 1: verification of Lemma \ref{lem:decomp}}
This test verifies Lemma \ref{lem:decomp}, or \eqref{eq:local-harmonic}. We compare
\[
\widehat{\sigma}:=\max_{i=1,\ldots,N}
\frac{\left\|Q_i v_h-v_h|_{\Omega_i}\right\|_{V(\Omega_i)}}
{\|v_h\|_{V(\widetilde{\Omega}_{i})}}
\]
with the scaling factor ${(hk)}^{-1/2}$ in \eqref{eq:local-harmonic}; the results are presented in Table \ref{ta:test1}.
The correlation coefficient between $\widehat{\sigma}$ and ${(hk)}^{-1/2}$ is 0.9999 in {\em Setting A}. In {\em Setting B}, both quantities remain constant over the tested wavenumbers. These observations are consistent with the dependence on $hk$ in Lemma \ref{lem:decomp}.
	
\begin{table}[hbt!]
		\centering 
        \begin{adjustbox}{max width=\textwidth}
        \begin{tabular}{ccccccccc}
			\toprule 
			\multicolumn{5}{c}{$h\simeq k^{-3/2}$, $H=1/k$}\tabularnewline
			\midrule 
			$k$& $h$ & $H$ & $ {(hk)}^{-\frac{1}{2}}$& $\widehat{\sigma}$\tabularnewline\hline
			10&$1/40$&1/10&  2.00&  1.40  \tabularnewline\hline % 8.0000
			20&$1/100$&1/20& 2.24  &  1.52  \tabularnewline\hline%8.9443
			40&$1/280$&1/40&2.65   &  1.72  \tabularnewline\hline%10.5830
			80&$1/720$&1/80& 3.00  &   1.90 \tabularnewline\hline%12.0000
			160&1/2080 &1/160&  3.61 &  2.19  \tabularnewline\hline%14.4222
			320&1/5760&1/320&  4.24 &  2.49  \tabularnewline  %%%16.9706 $ (1+kH)(1+(kH)^{-1})(hk)^{-1/2}$ 
	\bottomrule
        \end{tabular} \qquad  
        \begin{tabular}{ccccccccc}
        \toprule
			\multicolumn{5}{c}{$h= 1/(3k)$, $H=2/k$}\tabularnewline\midrule
			$k$& $h$ & $H$ &  $ {(hk)}^{-\frac{1}{2}}$& $\widehat{\sigma}$\tabularnewline\hline
			10&1/30& 1/5& 1.73 &1.25   \tabularnewline\hline
			20&1/60& 1/10& 1.73 & 1.25   \tabularnewline\hline
			40&1/120& 1/20&1.73 & 1.25  \tabularnewline\hline
			80&1/240& 1/40& 1.73 & 1.25   \tabularnewline\hline
			160&1/480& 1/80& 1.73 & 1.25  \tabularnewline\hline
			320&1/960&1/160 & 1.73 & 1.25 \tabularnewline	\bottomrule		%%  6.9282
		\end{tabular}
        \end{adjustbox}
		\caption{Comparison of computed value $\widehat{\sigma}$ with predicted value $ {(hk)}^{-\frac{1}{2}}$ for various $h$.}
		\label{ta:test1}
	\end{table}

\begin{table}[h!]
\centering
\begin{adjustbox}{max width=0.95\textwidth}
\begin{tabular}{ccccccccc}
\toprule 
\multicolumn{9}{c}{$h\simeq k^{-3/2}, H=1/k$}\tabularnewline
\midrule 
				\multirow{2}{*}{$k$}&  \multirow{2}{*}{$h$}&  \multirow{2}{*}{$H$}& \multicolumn{2}{c}{$\ell=0$} & \multicolumn{2}{c}{$\ell=1$} & \multicolumn{2}{c}{$\ell=2$}\tabularnewline
				\cmidrule{4-9} \cmidrule{5-9} \cmidrule{6-9} \cmidrule{7-9} \cmidrule{8-9} \cmidrule{9-9}
				&&	&$\Delta$& $2^{-\ell/2}\bar{d}$ & $\Delta$& $2^{-\ell/2}\bar{d}$ &$\Delta$& $2^{-\ell/2}\bar{d}$\tabularnewline
				\midrule 
				10&1/40& 1/10&  0.014488  &0.1000   & 0.012014 & 0.0707 &0.010737 & 0.0500\tabularnewline	\midrule 
				20&1/100&1/20 &  0.005940 &0.0500   &0.005096 & 0.0354 &0.004565 & 0.0250\tabularnewline\midrule
				40&1/280&1/40 &  0.002194  & 0.0250   & 0.001965 & 0.0177 &0.001720 & 0.0125\tabularnewline\midrule
				80&1/720& 1/80&  0.000871 & 0.0125  & 0.000801 & 0.0088 &0.000704 & 0.0063\tabularnewline\midrule
				160&1/2080&1/160 & 0.000309 & 0.0063   & 0.000293 & 0.0044 &0.000263 & 0.0031\tabularnewline\midrule
				320&1/5760&1/320 & 0.000114 &  0.0031   &  0.000110 & 0.0022& 0.000098 & 0.0016\tabularnewline
				\bottomrule
				
				\multicolumn{9}{c}{$h= 1/(3k)$, $H=2/k$}\tabularnewline
				\midrule 
				\multirow{2}{*}{$k$}&  \multirow{2}{*}{$h$}&  \multirow{2}{*}{$H$}& \multicolumn{2}{c}{$\ell=0$} & \multicolumn{2}{c}{$\ell=1$} & \multicolumn{2}{c}{$\ell=2$}\tabularnewline
				\cmidrule{4-9} \cmidrule{5-9} \cmidrule{6-9} \cmidrule{7-9} \cmidrule{8-9} \cmidrule{9-9}
				&&	&$\Delta$& $2^{-\ell/2}\bar{d}$ & $\Delta$& $2^{-\ell/2}\bar{d}$ &$\Delta$& $2^{-\ell/2}\bar{d}$\tabularnewline
				\midrule 
				10&1/30& 1/5 & 0.0219196 & 0.200000   & 0.0181578 & 0.141421  & 0.0147435 & 0.100000 \tabularnewline	\midrule 
				20&1/60& 1/10&  0.0109598 & 0.100000 & 0.0090789 & 0.070711  & 0.0073717 & 0.050000\tabularnewline\midrule
				40&1/120& 1/20& 0.0054799 & 0.050000   & 0.0045394 & 0.035355  & 0.0036859 & 0.025000\tabularnewline\midrule
				80&1/240& 1/40& 0.0027399 & 0.025000  & 0.0022697 & 0.017678  &0.0018429 & 0.012500\tabularnewline\midrule
				160&1/480&1/80 & 0.0013700 & 0.012500   & 0.0011349 & 0.008839  & 0.0009215 & 0.006250\tabularnewline\midrule
				320&1/960& 1/160& 0.0006850 & 0.006250   & 0.0005674 & 0.004419  & 0.0004607 & 0.003125\tabularnewline
				\bottomrule
				
			\end{tabular}
		\end{adjustbox}
		\caption{Comparison of computed value $\Delta$ with $2^{-\ell/2}\bar{d}$. }
			\label{ta:test2a}
	\end{table}
\subsection{Test 2: verification of Corollary \ref{thm:proj}}
Now we verify the estimates in Corollary \ref{thm:proj}. We only need to verify $\Delta:=\max_{i=1,\ldots,N}\frac{ \| e_i\|_{L^2(\Omega_i)} }{ \|  v_i^{\partial}\|_{H^1(\Omega_i)} }$, since the estimate on $\max_{i=1,\ldots,N}\frac{ \|\chi_i \nabla e_i\|_{L^2(\Omega_i)} }{ \|v_i^{\partial}\|_{H^1(\Omega_i)} }$ is derived from $\Delta$.
Table \ref{ta:test2a} compares $\Delta$ with the factor $2^{-\ell/2}\bar{d}$ in the bound of Corollary \ref{thm:proj}. The observed dependence of $\Delta$ on $\bar{d}$ is approximately linear in both {\em Setting A} and {\em Setting B}, consistent with this bound.
\subsection{Test 3: challenging fixed-\texorpdfstring{$n_{\rm ppw}$}{nppw} family in two dimensions}
To assess how trace enrichment affects convergence as the wavenumber increases, we consider \eqref{hel_cont} on $\Omega=(0,1)^2$ with constant $k$, homogeneous impedance data $g=0$, and the discrete point source
\[
F_h(v_h)=\overline{v_h}(1/2,1/2).
\]
We increase $k$ while refining the uniform Cartesian mesh so that
\[
n_{\rm ppw}:=\frac{2\pi}{kh}=20,
\qquad \frac{H}{h}=\frac{\delta}{h}=16,
\]
This keeps the number of mesh intervals per wavelength and per subdomain fixed, while increasing the number of subdomains. Comparing $\ell=0$ and $\ell=1$ therefore tests whether the corresponding coarse spaces maintain convergence as the global wave propagation problem grows. The discretization uses conforming bilinear elements with tensor-product Gauss quadrature and a normalized cubic partition of unity; the same choices, with trilinear elements, are used in the three-dimensional experiments below.

We compare one-level RAS-imp and EMs-HS using right-preconditioned BiCGStab, and include stationary EMs-HS to assess convergence without Krylov acceleration. For a common measure of convergence in this and the subsequent experiments, we report the number of preconditioner applications required from a zero initial guess to satisfy the directly evaluated relative residual criterion
\[
\frac{\|Au_h-F_h\|_2}{\|F_h\|_2}\leq 10^{-6}.
\]

Table \ref{ta:fixed-ppw-2d} shows that enriching the trace space changes the dependence of convergence on the wavenumber. Whereas the one-level counts grow with $k$, level-$1$ EMs-HS requires only five to seven applications with BiCGStab and six to nine without it. The level-$0$ space does not maintain this behavior: stationary iteration diverges on the three finer meshes, and even BiCGStab fails to converge within 200 applications on the finest mesh. In this family, Krylov acceleration improves convergence, but the richer trace space is needed to retain low application counts across the full wavenumber range.

\begin{table}[hbt!]
\centering
\begin{tabular}{ccccccc}
\toprule
 & & RAS-imp & \multicolumn{2}{c}{EMs-HS, $\ell=0$} & \multicolumn{2}{c}{EMs-HS, $\ell=1$}\\
\cmidrule(lr){4-5}\cmidrule(lr){6-7}
$1/h$ & $k$ & BiCGStab & Stationary & BiCGStab & Stationary & BiCGStab\\
\midrule
64  & 20.11  & 7  & 9 & 9 & 6 & 5\\
128 & 40.21  & 18 & div. & 16 & 9 & 7\\
256 & 80.42  & 42 & div. & 30 & 8 & 6\\
512 & 160.85 & 88 & div. & $>200$ & 8 & 6\\
\bottomrule
\end{tabular}
\caption{Preconditioner applications in two dimensions with $n_{\rm ppw}=20$ and $H/h=\delta/h=16$. Here \textup{div.} denotes divergence, and $>200$ denotes failure to reach the stopping criterion within 200 applications.}
\label{ta:fixed-ppw-2d}
\end{table}

\subsection{Test 4: challenging fixed-\texorpdfstring{$n_{\rm ppw}$}{nppw} family in three dimensions}
The three-dimensional experiment examines whether low application counts persist as the number of unknowns increases to more than $135$ million. We consider the homogeneous problem on $\Omega=(0,1)^3$, with $g=0$ and a discrete point source at the centre, and refine the mesh with
\[
n_{\rm ppw}=20,\qquad \frac{H}{h}=\frac{\delta}{h}=8,\qquad \ell=0.
\]
As in the two-dimensional family, this refinement increases both the wavenumber and the number of subdomains at fixed wavelength resolution. Table \ref{ta:fixed-ppw-3d} shows that BiCGStab requires four preconditioner applications at every mesh size, while stationary EMs-HS requires four or five. Thus convergence remains rapid throughout this three-dimensional family even without Krylov acceleration.

\begin{table}[hbt!]
\centering
\begin{tabular}{ccccc}
\toprule
$1/h$ & $\dim V_h$ & $k$ & Stationary & BiCGStab\\
\midrule
64  & 274,625     & 20.11 & 4 & 4\\
128 & 2,146,689   & 40.21 & 4 & 4\\
256 & 16,974,593  & 80.42 & 4 & 4\\
512 & 135,005,697 & 160.85 & 5 & 4\\
\bottomrule
\end{tabular}
\caption{Preconditioner applications in three dimensions with $n_{\rm ppw}=20$, $H/h=\delta/h=8$, and $\ell=0$.}
\label{ta:fixed-ppw-3d}
\end{table}

\subsection{Test 5: heterogeneous media}
We now examine EMs-HS for the three heterogeneous media shown in Figures~\ref{fig:hetero-A}--\ref{fig:scatter200}. Models A and B are two-dimensional: Model A contains high-speed inclusions with $c=1.4$ in a background with $c=1$, while Model B contains low-speed fractures with $c=1$ in a background with $c=1.4$. Model C is the three-dimensional medium with intersecting high-speed inclusions. The domains are the unit square for Models A and B and the unit cube for Model C. Their cellwise velocity distributions are represented on $1000^2$, $800^2$, and $200^3$ Cartesian meshes, respectively. We take angular frequencies $\omega=80$, $100$, and $29$, giving the local wavenumber ranges listed in Table~\ref{ta:heterogeneous} through $k(x)=\omega/c(x)$.

Each medium is excited by a localized Gaussian source
\[
f(x)=\exp\bigl(-|x-x_s|^2/h^2\bigr),
\]
with homogeneous impedance data. The source is centred at $x_s=(0.5,0.5)$ in two dimensions and at $x_s=(0.5,0.5,0.1)$ in three dimensions. In Models A and B its nodal interpolant defines the finite element load; in Model C it is sampled at cell centres and represented elementwise. Using $H/h=8$ and $\delta=H$ throughout, we compare $\ell=0$ and $\ell=1$ in two dimensions and retain $\ell=0$ in three dimensions.

Table~\ref{ta:heterogeneous} shows that level-$0$ EMs-HS converges in five BiCGStab preconditioner applications for all three media. Trace enrichment reduces this count to four in both two-dimensional models. The low application counts thus persist across these inclusion and fracture geometries, including the three-dimensional system with $8{,}120{,}601$ unknowns. The corresponding wave fields are displayed alongside the velocities in Figures~\ref{fig:hetero-A}--\ref{fig:scatter200}.

\begin{table}[htbp]
\centering
\begin{tabular}{ccccc}
\toprule
Model & $\dim V_h$ & Range of $k(x)$ & $\ell=0$ & $\ell=1$\\
\midrule
A & 1,002,001 & $57.14\sim 80$ & 5 & 4\\
B & 641,601 & $71.43\sim 100$ & 5 & 4\\
C & 8,120,601 & $20.71\sim 29$ & 5 & /\\
\bottomrule
\end{tabular}
\caption{BiCGStab preconditioner applications for heterogeneous media with $H/h=\delta/h=8$. The symbol `/' denotes a case not tested.}
\label{ta:heterogeneous}
\end{table}

\begin{figure}[htbp]
\centering
\includegraphics[width=0.8\textwidth]{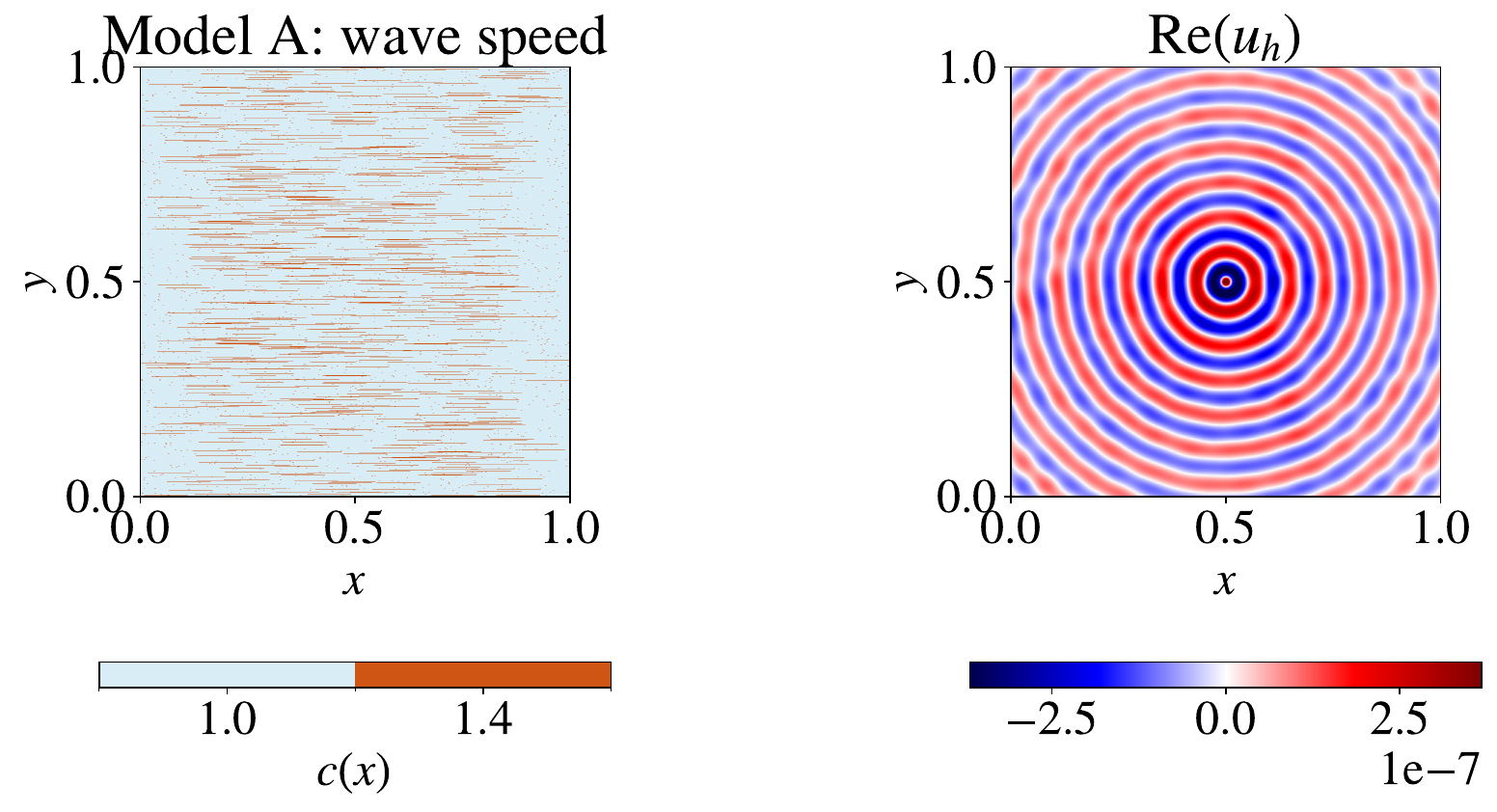}
\caption{Model A: wave speed and real part of the solution computed with $\ell=1$. The symmetric color limits are set by the $99.8$th percentile of $|\mathrm{Re}(u_h)|$; values beyond these limits saturate.}
\label{fig:hetero-A}
\end{figure}

\begin{figure}[htbp]
\centering
\includegraphics[width=0.8\textwidth]{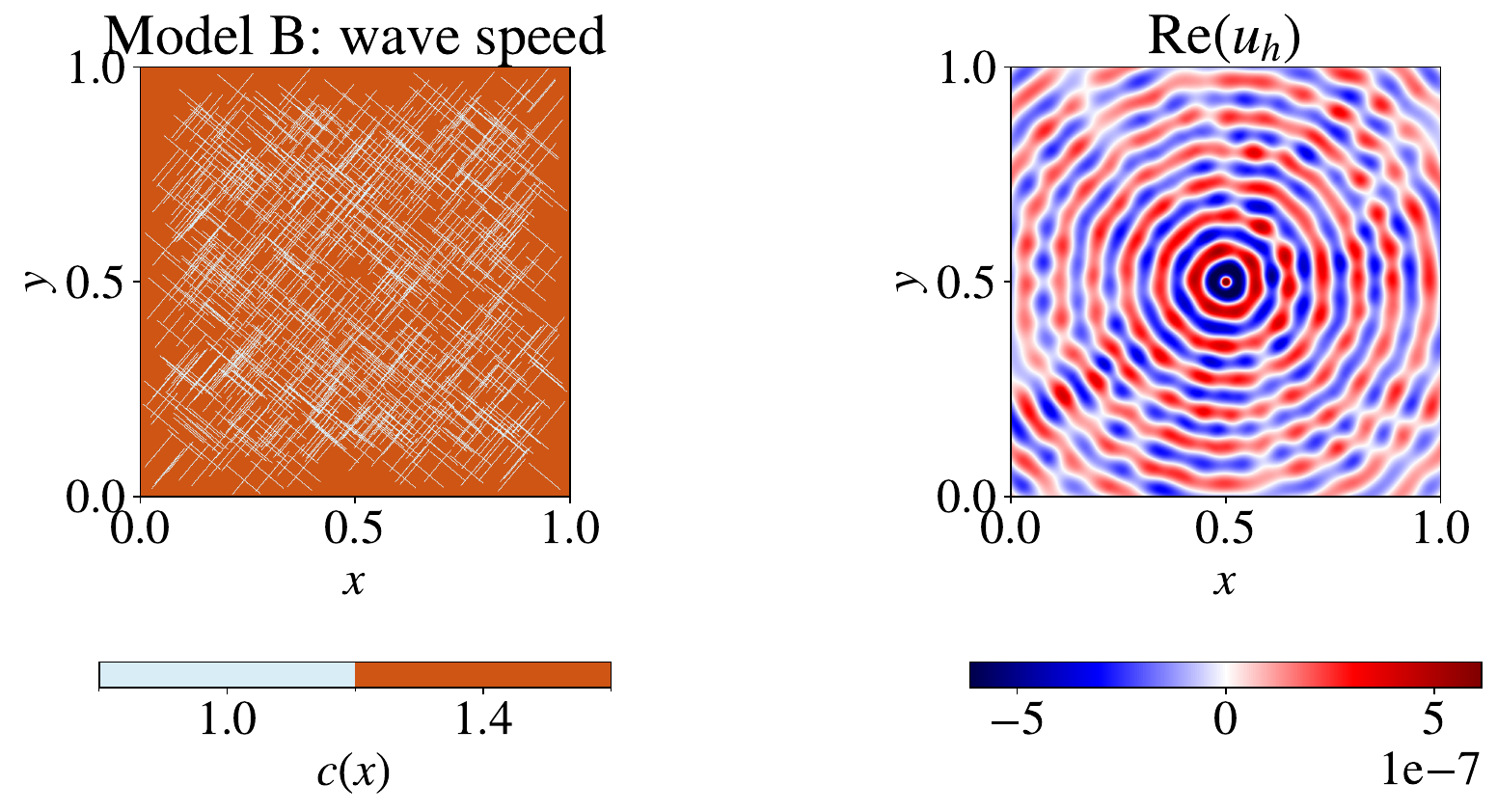}
\caption{Model B: wave speed and real part of the solution computed with $\ell=1$, using the same color-limit convention as Figure~\ref{fig:hetero-A}.}
\label{fig:hetero-B}
\end{figure}

\begin{figure}[htbp]
\centering
\includegraphics[width=0.8\textwidth]{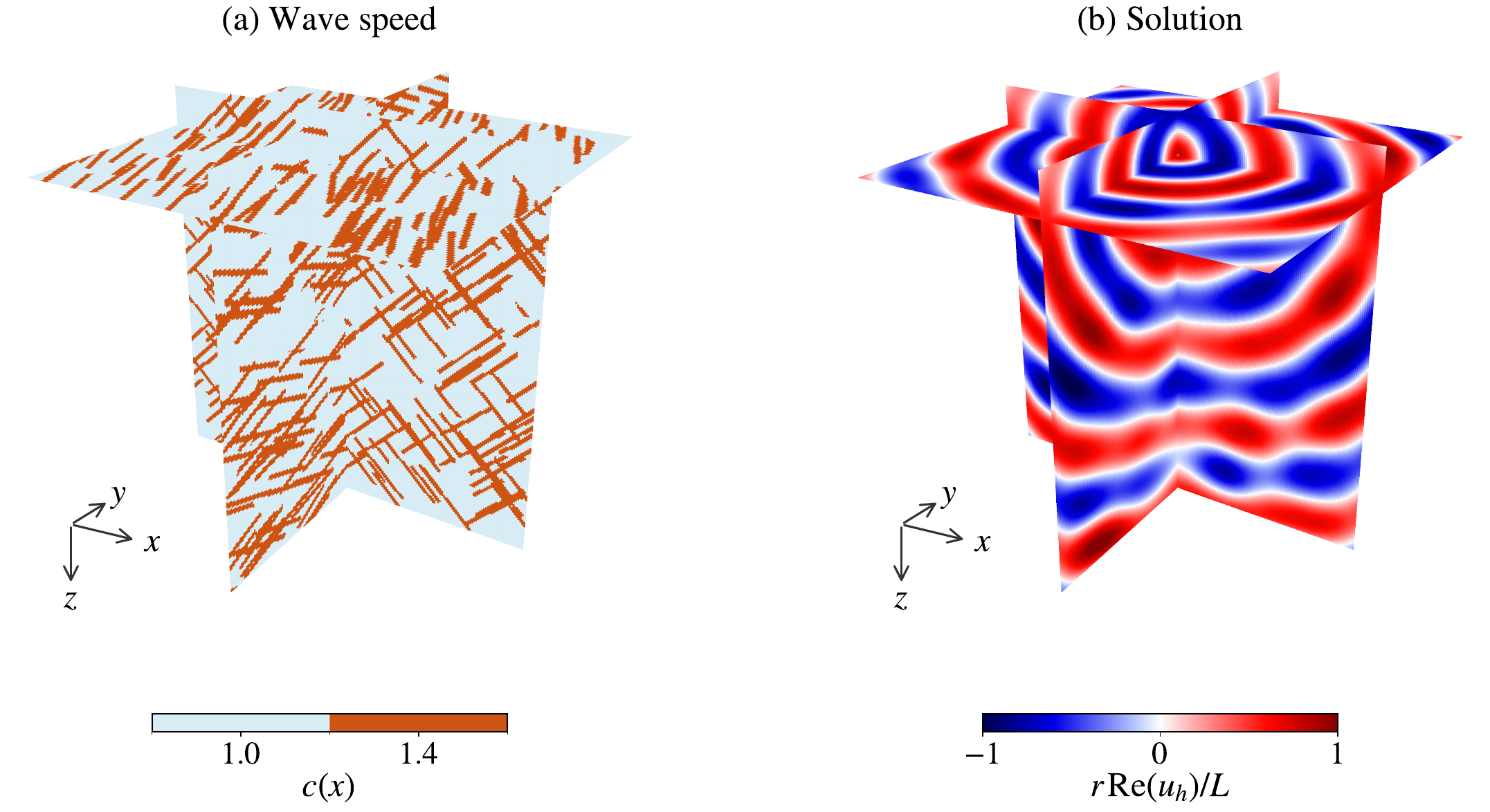}
\caption{Model C: (a) wave speed and (b) computed solution on $x=0.5$, $y=0.5$, and $z=0.1$, with $z$ displayed downward. To reveal oscillations away from the source, panel (b) displays $r\,\mathrm{Re}(u_h)/L$, where $r=|x-x_s|$ and $L$ is the $99.7$th percentile of $|r\,\mathrm{Re}(u_h)|$ on the plotted sections. Colors saturate outside $[-1,1]$.}
\label{fig:scatter200}
\end{figure}

 \section{Conclusion}\label{sec:conclusion}
 We have designed a novel hybrid Schwarz preconditioner for the Helmholtz problem with large wavenumbers based on the edge multiscale ansatz space, which allows for overlapping size of one coarse grid. We have rigorously analyzed its convergence, and demonstrated the performance of the method with extensive numerical tests in both 2-d and 3-d. There are several interesting questions deserving further investigations. First, we want to design scalable solvers for the coarse problem, of which the dimension is increasing with respect to the wavenumber $k$.  Second, we want to design a similar hybrid Schwarz preconditioner for the time-harmonic Maxwell equation with large wavenumber. 
\section*{Acknowledgement}
SG acknowledges the support from NSFC (No. 12201535), Guangdong Basic and Applied Basic Research Foundation (No. 2023A1515011651), and  Shenzhen Stability Science Program 2022. GL acknowledges the support from Young Scientists fund (No.: 12101520) by NSFC and General Research Fund (No.: 17308924) of Hong Kong RGC.

\bibliographystyle{plain}
\bibliography{reference}

\begin{thebibliography}{10}

\bibitem{MR2084239}
Martin Berggren.
\newblock Approximations of very weak solutions to boundary-value problems.
\newblock {\em SIAM J. Numer. Anal.}, 42(2):860--877, 2004.

\bibitem{MR4621835}
Niall Bootland, Victorita Dolean, Ivan~G. Graham, Chupeng Ma, and Robert Scheichl.
\newblock Overlapping {S}chwarz methods with {G}en{EO} coarse spaces for indefinite and nonself-adjoint problems.
\newblock {\em IMA J. Numer. Anal.}, 43(4):1899--1936, 2023.

\bibitem{MR4299039}
Niall Bootland, Victorita Dolean, Pierre Jolivet, and Pierre-Henri Tournier.
\newblock A comparison of coarse spaces for {H}elmholtz problems in the high frequency regime.
\newblock {\em Comput. Math. Appl.}, 98:239--253, 2021.

\bibitem{cai1992domain}
Xiao-Chuan Cai and Olof~B Widlund.
\newblock Domain decomposition algorithms for indefinite elliptic problems.
\newblock {\em SIAM J. Sci. Stat. Comput.}, 13(1):243--258, 1992.

\bibitem{MR520174}
Philippe~G. Ciarlet.
\newblock {\em The finite element method for elliptic problems}.
\newblock North-Holland Publishing Co., Amsterdam-New York-Oxford, 1978.

\bibitem{MR1162107}
Ingrid Daubechies.
\newblock {\em Ten Lectures on Wavelets}.
\newblock SIAM, Philadelphia, PA, 1992.

\bibitem{Graham-geneo-2024}
Victorita Dolean, Mark Fry, Ivan Graham, and Matthias Langer.
\newblock Schwarz preconditioner with $h_k$-{G}en{EO} coarse space for the indefinite {H}elmholtz problem.
\newblock {\em arXiv preprint arXiv:2406.06283}, 2024.

\bibitem{engquist2011sweepinga}
Bj{\"o}rn Engquist and Lexing Ying.
\newblock Sweeping preconditioner for the {H}elmholtz equation: hierarchical matrix representation.
\newblock {\em Commun. Pure Appl. Math.}, 64(5):697--735, 2011.

\bibitem{erlangga2006novel}
Yogi~A Erlangga, Cornelis~W Oosterlee, and Cornelis Vuik.
\newblock A novel multigrid based preconditioner for heterogeneous {H}elmholtz problems.
\newblock {\em SIAM J. Sci. Comput.}, 27(4):1471--1492, 2006.

\bibitem{erlangga2004class}
Yogi~A Erlangga, Cornelis Vuik, and Cornelis~Willebrordus Oosterlee.
\newblock On a class of preconditioners for solving the {H}elmholtz equation.
\newblock {\em Appl. Numer. Math.}, 50(3-4):409--425, 2004.

\bibitem{MR3980476}
Shubin Fu, Eric Chung, and Guanglian Li.
\newblock Edge multiscale methods for elliptic problems with heterogeneous coefficients.
\newblock {\em J. Comput. Phys.}, 396:228--242, 2019.

\bibitem{MR4860976}
Shubin Fu, Eric Chung, and Guanglian Li.
\newblock Wavelet-based edge multiscale finite element methods for singularly perturbed convection-diffusion equations.
\newblock {\em Multiscale Model. Simul.}, 23(1):431--457, 2025.

\bibitem{fu2021wavelet}
Shubin Fu, Guanglian Li, Richard Craster, and Sebastien Guenneau.
\newblock Wavelet-based edge multiscale finite element method for {H}elmholtz problems in perforated domains.
\newblock {\em Multiscale Model. Simul.}, 19(4):1684--1709, 2021.

\bibitem{MR5013301}
J.~Galkowski and E.~A. Spence.
\newblock Convergence {T}heory for {T}wo-{L}evel {H}ybrid {S}chwarz {P}reconditioners for {H}igh-{F}requency {H}elmholtz {P}roblems.
\newblock {\em SIAM J. Numer. Anal.}, 64(1):29--54, 2026.

\bibitem{galkowski2024convergence}
Jeffrey Galkowski, Shihua Gong, Ivan~G Graham, David Lafontaine, and Euan~A Spence.
\newblock Convergence of overlapping domain decomposition methods with {PML} transmission conditions applied to nontrapping {H}elmholtz problems.
\newblock {\em arXiv preprint arXiv:2404.02156}, 2024.

\bibitem{MR2718268}
Juan Galvis and Yalchin Efendiev.
\newblock Domain decomposition preconditioners for multiscale flows in high-contrast media.
\newblock {\em Multiscale Model. Simul.}, 8(4):1461--1483, 2010.

\bibitem{MR2728702}
Juan Galvis and Yalchin Efendiev.
\newblock Domain decomposition preconditioners for multiscale flows in high contrast media: reduced dimension coarse spaces.
\newblock {\em Multiscale Model. Simul.}, 8(5):1621--1644, 2010.

\bibitem{gong2022convergence}
Shihua Gong, Martin~J Gander, Ivan~G Graham, David Lafontaine, and Euan~A Spence.
\newblock Convergence of parallel overlapping domain decomposition methods for the {H}elmholtz equation.
\newblock {\em Numer. Math.}, 152(2):259--306, 2022.

\bibitem{gong2023convergence}
Shihua Gong, Ivan Graham, and Euan Spence.
\newblock Convergence of restricted additive {S}chwarz with impedance transmission conditions for discretised {H}elmholtz problems.
\newblock {\em Math. Comput.}, 92(339):175--215, 2023.

\bibitem{gong2021domain}
Shihua Gong, Ivan~G Graham, and Euan~A Spence.
\newblock Domain decomposition preconditioners for high-order discretizations of the heterogeneous {H}elmholtz equation.
\newblock {\em IMA J. Numer. Anal.}, 41(3):2139--2185, 2021.

\bibitem{graham2017domain}
Ivan Graham, Euan Spence, and Eero Vainikko.
\newblock Domain decomposition preconditioning for high-frequency {H}elmholtz problems with absorption.
\newblock {\em Math. Comput.}, 86(307):2089--2127, 2017.

\bibitem{graham2020domain}
Ivan~G Graham, Euan~A Spence, and Jun Zou.
\newblock Domain decomposition with local impedance conditions for the {H}elmholtz equation with absorption.
\newblock {\em SIAM J. Numer. Anal.}, 58(5):2515--2543, 2020.

\bibitem{MR775683}
P.~Grisvard.
\newblock {\em Elliptic problems in nonsmooth domains}, volume~24 of {\em Monographs and Studies in Mathematics}.
\newblock Pitman (Advanced Publishing Program), Boston, MA, 1985.

\bibitem{MR4890903}
Qiya Hu and Ziyi Li.
\newblock A hybrid two-level weighted {S}chwarz method for {H}elmholtz equations.
\newblock {\em SIAM J. Numer. Anal.}, 63(2):716--743, 2025.

\bibitem{ihlenburg1995dispersion}
Frank Ihlenburg and Ivo Babu{\v{s}}ka.
\newblock Dispersion analysis and error estimation of {G}alerkin finite element methods for the {H}elmholtz equation.
\newblock {\em Int. J. Numer. Methods Eng.}, 38(22):3745--3774, 1995.

\bibitem{ihlenburg1995finite}
Frank Ihlenburg and Ivo Babu{\v{s}}ka.
\newblock Finite element solution of the {H}elmholtz equation with high wave number part {I}: The h-version of the {FEM}.
\newblock {\em Comput. Math. Appl.}, 30(9):9--37, 1995.

\bibitem{kimn2007restricted}
Jung-Han Kimn and Marcus Sarkis.
\newblock Restricted overlapping balancing domain decomposition methods and restricted coarse problems for the {H}elmholtz problem.
\newblock {\em Comput. Methods Appl. Mech. Eng.}, 196(8):1507--1514, 2007.

\bibitem{MR3834684}
Ralf Kornhuber, Daniel Peterseim, and Harry Yserentant.
\newblock An analysis of a class of variational multiscale methods based on subspace decomposition.
\newblock {\em Math. Comput.}, 87(314):2765--2774, 2018.

\bibitem{MR3939320}
Guanglian Li.
\newblock On the convergence rates of {GM}s{FEM}s for heterogeneous elliptic problems without oversampling techniques.
\newblock {\em Multiscale Model. Simul.}, 17(2):593--619, 2019.

\bibitem{MR0350177}
J.-L. Lions and E.~Magenes.
\newblock {\em Non-homogeneous {B}oundary {V}alue {P}roblems and {A}pplications. {V}ol. {I}}.
\newblock Springer-Verlag, New York-Heidelberg, 1972.

\bibitem{MR4980418}
Peipei Lu, Xuejun Xu, Bowen Zheng, and Jun Zou.
\newblock Two-{L}evel {H}ybrid {S}chwarz {P}reconditioners for the {H}elmholtz {E}quation with {H}igh {W}ave {N}umber.
\newblock {\em SIAM J. Numer. Anal.}, 63(6):2187--2220, 2025.

\bibitem{MR4601686}
Chupeng Ma, Christian Alber, and Robert Scheichl.
\newblock Wavenumber explicit convergence of a multiscale generalized finite element method for heterogeneous {H}elmholtz problems.
\newblock {\em SIAM J. Numer. Anal.}, 61(3):1546--1584, 2023.

\bibitem{MR4992072}
Chupeng Ma, Christian Alber, Robert Scheichl, and Yongwei Zhang.
\newblock Two-{L}evel {R}estricted {A}dditive {S}chwarz {P}reconditioner {B}ased on {M}ultiscale {S}pectral {G}eneralized {FEM} for {H}eterogeneous {H}elmholtz {P}roblems.
\newblock {\em J. Sci. Comput.}, 105(3):Paper No. 99, 2025.

\bibitem{MR1742312}
William McLean.
\newblock {\em Strongly elliptic systems and boundary integral equations}.
\newblock Cambridge University Press, Cambridge, 2000.

\bibitem{melenk2010convergence}
J~Melenk and Stefan Sauter.
\newblock Convergence analysis for finite element discretizations of the {H}elmholtz equation with {D}irichlet-to-{N}eumann boundary conditions.
\newblock {\em Math. Comput.}, 79(272):1871--1914, 2010.

\bibitem{MR2812565}
J.~M. Melenk and S.~Sauter.
\newblock Wavenumber explicit convergence analysis for {G}alerkin discretizations of the {H}elmholtz equation.
\newblock {\em SIAM J. Numer. Anal.}, 49(3):1210--1243, 2011.

\bibitem{melenk1995generalized}
Jens~Markus Melenk.
\newblock {\em On generalized finite element methods}.
\newblock PhD thesis, research directed by Dept. of Mathematics.University of Maryland at College Park, 1995.

\bibitem{MR3614010}
Daniel Peterseim.
\newblock Eliminating the pollution effect in {H}elmholtz problems by local subscale correction.
\newblock {\em Math. Comput.}, 86(305):1005--1036, 2017.

\bibitem{sauter2006refined}
Stefan~A Sauter.
\newblock A refined finite element convergence theory for highly indefinite {H}elmholtz problems.
\newblock {\em Computing}, 78:101--115, 2006.

\bibitem{MR373326}
Alfred~H. Schatz.
\newblock An observation concerning {R}itz-{G}alerkin methods with indefinite bilinear forms.
\newblock {\em Math. Comp.}, 28:959--962, 1974.

\bibitem{toselli2006domain}
Andrea Toselli and Olof Widlund.
\newblock {\em Domain decomposition methods-algorithms and theory}, volume~34.
\newblock Springer Science \& Business Media, 2006.

\bibitem{MR3232452}
Haijun Wu.
\newblock Pre-asymptotic error analysis of {CIP}-{FEM} and {FEM} for the {H}elmholtz equation with high wave number. {P}art {I}: linear version.
\newblock {\em IMA J. Numer. Anal.}, 34(3):1266--1288, 2014.

\bibitem{xu1992new}
Jinchao Xu.
\newblock A new class of iterative methods for nonself adjoint or indefinite problems.
\newblock {\em SIAM J. Numer. Anal.}, 29(2):303--319, 1992.

\end{thebibliography}

%USE THE BELOW OPTIONS IN CASE YOU NEED AUTHOR YEAR FORMAT.
%\bibliographystyle{abbrvnat}
%\bibliography{reference}

\end{document}